\documentclass{m2an}

\usepackage[utf8]{inputenc}
\usepackage{graphicx,graphics}

\usepackage{array}
\usepackage{appendix}
\usepackage{amsmath,amssymb}

\usepackage{epstopdf}

\usepackage{enumerate}
\usepackage{mathtools}
\usepackage{color,xcolor} 
\usepackage{mathrsfs,bbm} 

\usepackage{hyperref}
\usepackage{amsthm}

\usepackage{mathpazo}
\usepackage{dsfont}

\usepackage{epstopdf}

\usepackage{algorithm}
\usepackage{algorithmic}
\usepackage{enumerate}
\usepackage{mathtools}
\usepackage{xcolor} 

\usepackage{mathrsfs} 
\usepackage{commath}

\newcommand{\eps}{\varepsilon}

\renewcommand{\Re}{{\mathrm{Re}}\,}
\renewcommand{\Im}{{\mathrm{Im}}\,}

\definecolor{refblue}{rgb}{0,0,0.75} 
\definecolor{refblueb}{rgb}{0,0,1} 
\definecolor{refgreen}{rgb}{0,0.5,0} 
\definecolor{grey}{rgb}{0.7,0.7,0.7} 

\usepackage{hyperref}
\hypersetup{ 
	colorlinks   = true, 
	urlcolor     = refblue, 
	linkcolor    = refblueb, 
	citecolor   = refblue
}

\newtheorem{proposition}{Proposition}
\newtheorem{lemma}{Lemma}

\newtheorem{theorem}{Theorem}

\theoremstyle{custom1}
\newtheorem{definition}{Definition}[section]
\newtheorem{remark}[definition]{Remark}
\theoremstyle{custom2}

\begin{document}
\title[The MFE for waves propagating through time-modulated media]{The modulated Fourier expansion for waves propagating through time-modulated media}
%
\author{Jörg Nick}\address{Section of Mathematics, University of Geneva, Switzerland.}

%
%
\begin{abstract} Controlling waves by actively changing the material parameters of a medium  
enables the development of new acoustic and electrical devices. Modulating the material breaks classical properties like reciprocity and the conservation of energy, which complicates the mathematical analysis. Without a limiting amplitude principle, time-harmonic formulations are generally inapplicable. The present manuscript develops an alternative tool for the time-modulated acoustic wave equation, that is based on a modulated Fourier expansion (MFE). The solution is characterized by multiple smoothly varying coefficient functions, which solve a coupled system of evolutionary partial differential equations with temporally constant coefficients.
For small-amplitude fast-time modulations, this system of evolutionary partial differential equations is shown to possess a smoothly varying solution, which characterizes the exact solution up to a small defect.

Discretization of the derived coupled system yields integrators that are stable and accurate when larger time steps are used, compared to those schemes that are applied to the time-modulated acoustic wave equation directly. Numerical experiments illustrate the theoretical results and the use of the approach.

\end{abstract}
%
%
\maketitle

\section{Introduction}
The promise of an additional degree of freedom in the design of materials, namely the variation of the physical parameters in time, has lead to the exploration of new types of metamaterials. Time-varying materials allow the construction of devices that are not achievable by passive media, as waves propagation through time-modulated media are not bound by reciprocity and the conservation of energy. These properties are central in the development of non-reciprocal devices such as insulators or circulators \cite{S19}, or to achieve signal amplification \cite{A22}. Other applications include frequency conversion \cite{Y17} as the modulation of the material couples various monochromatic components of incoming and scattered waves to each other, which enables optical and acoustic devices that manipulate the spectrum of waves propagating through time-modulated media in new ways. Consequently, these novel materials have received much attention in recent years \cite{YGA22}. 

Acoustic waves that traverse time-varying media are modeled by the acoustic wave equation with time-varying coefficients, which generally reads
\begin{align}\label{eq:ac-intro}
	 \partial_t^2 u(x,t) - \nabla \cdot \mu(x,t/\eps)\nabla u(x,t) = f (x,t), \quad \text{in} \quad \Omega \times [0,T].
\end{align}
The formulation is completed by homogeneous Dirichlet boundary conditions, the final time is denoted by $T$ and $\Omega \subset \mathbb{R}^d$ is assumed to be a bounded Lipschitz domain of dimension $d\in\{1,2,3\}$.
Finally, the solution $u$ is assumed to initially vanish and is excited by the smooth source term $f$, which also vanishes initially. The modulation $\mu$ is assumed to be $2\pi$ periodic and has the form
\begin{align*}
\mu(x,t/\varepsilon) = \mu_0(x)+ \rho\mu_{\mathrm{per}}(x,t/\eps),
\end{align*}
where $\rho$ is an amplitude that is assumed to be small, since stability bounds for the evolution problem \eqref{eq:ac-intro} are, to the knowledge of the author, only known to be stable up to factors of the type $e^{C\rho t /\eps}$ for some constant $C>0$.

\subsection{Related work}
Wave equations with variable wave speeds with \emph{bounded oscillations} have long been studied, following key results e.g. in \cite{CD03},\cite{RS05}, \cite{T07} and \cite{HW09}. Fast oscillations in dissipative wave equations have been investigated in \cite{RY00} and, very recently, in \cite{GH25}. Mathematical analysis for oscillatory modulations, which cover e.g. small values of $\varepsilon$, are much less present in the literature. 
 Homogenization techniques, as investigated in \cite{FA24,TLAGC24}, have shown promising results in computations but are currently supported only by limited rigorous theory (see \cite{FA24}). Moreover, high order homogenization is necessary to observe interesting behavior, such as nonreciprocity \cite{ACH22}. In \cite{DV25}, homogenized equations for time-varying materials are derived. Rigorous mathematical error estimates in the literature are, to the knowledge of the author, either accompanied by very restrictive non-resonance conditions or omitted \cite{FA24}. In \cite{G21}, the effect of periodic and random time-modulations on waves is quantified and compared. Other avenues to stable integrators for evolution problems with highly oscillatory coefficient include two-scale formulations \cite{CCNM15}. A collection of approaches is found in the monograph \cite{WYW13}.

The problem formulation of the time-modulated acoustic wave equation can be
completed in different ways. Studying the long-time behavior of waves propagating through (time-modulated) metamaterials naturally motivates the study of temporal quasi-periodic boundary conditions. This idea corresponds to Floquet--Bloch based approaches and can be understood as studying the spectrum of the forward operator that propagates the solution by one full period (i.e. $u(\cdot,t) \rightarrow u(\cdot,t+\eps 2\pi)$).  Floquet--Bloch theory is an established theory to analyze ordinary differential equations and elliptic partial differential equations with periodic coefficients (see e.g. \cite{K93}) and has further been used to investigate the Schrödinger equation with a time-periodic Hamiltonian \cite{TK19}. For a large part, these results cannot be carried over to the time-modulated acoustic wave equation. It is currently unknown whether core building blocks of such an approach, such as a counterpart to the Floquet theorem for ordinary differential equations (see \cite{K93}) or a limiting amplitude principle, can be formulated for time-modulated acoustic media. When the system is discretized in space, the Floquet theorem for ordinary differential equations does hold, but drawing conclusions by approaching the spatially continuous system by a sequence of refinements is challenging \cite{NHA24}. An intuitive approach to overcome these difficulties is to neglect (or filter) high temporal oscillation in the problem formulation and focus on a range of frequencies, which is an approach also known as harmonic balancing or coupled harmonics. This approach is difficult to connect directly to the general initial value problem associated to the time-modulated acoustic wave equation, but has been shown to be an effective tool to gain insight into the properties of the time-modulated materials, see e.g. \cite{ACHR24,HD24}. The recent manuscript \cite{HR25} quantifies energy changes for waves propagating through time-dependent media with high contrast.

Studying the modulated acoustic wave equation as an initial value problem, relatively large times have to be simulated and many oscillations of $\mu$ have to be resolved in order to draw conclusions on the long-time behavior of the evolution problem \eqref{eq:ac-intro}, even when modulation has small amplitude. The present paper attempts to reconcile these different approaches, by deriving a modulated Fourier expansion for the solution of the initial-value problem, that characterizes the solution by a few slowly varying functions. The approach taken here is therefore closely related to prior derivations of modulated Fourier expansions, which were already used in \cite{K58} and further developed e.g. in \cite{MV76} and \cite{HL00}. Their modern name and notation was coined in \cite{HLW02}. The modulated Fourier expansion makes the ansatz that 
\begin{align}\label{eq:modul-Fourier}
u(x,t) = \sum_{k=-K}^K z^K_k(x,t) e^{i kt/\varepsilon} + R_K(x,t),
\end{align}
where the coefficient functions $z^K_k$ for $-K\le k\le K $ are slowly varying, in the sense that their temporal derivatives are bounded independent of $\varepsilon$ and the remainder term $R_K$ is small. These expansions have been observed to be effective tools for many evolution problems and their discretizations. Most relevant for the present topic is their use for oscillatory ordinary differential equations \cite{CHL03}, for nonlinear wave equations formally in \cite{CHL08} and for semilinear wave equations with slowly varying wave speeds \cite{GHL16}. A concise overview of modulated Fourier expansions and their applications is found in \cite{GHL18}. The treatment of the wave equation with slowly varying wave speeds conducted in \cite{GHL16} is the closest to the the present work, however its technique critically relies on the fact that the modulation is slowly varying, and cannot be carried over to the present case. The core difference between the treatment in \cite{GHL16} and the present manuscript becomes apparent in the construction of the coefficient functions $z^K$. Whereas the slowly varying time-modulated system in \cite{GHL16} relies on a single dynamical system that is complemented by a system of time-harmonic problems, this manuscript derives the coefficients $z^K$ as the solution of a coupled system of evolution problems. 
\subsection{Contributions of this manuscript}
This manuscript derives and analyses a linear time-invariant system whose slowly varying solution $z^K$ characterizes, up to a small remainder term, the exact solution of the time-modulated acoustic wave equation. The coefficients are slowly varying in the sense that their time derivatives are bounded independently of the oscillation period $\eps$, as long as $\rho t< \eps$, where $t$ denotes the time of interest. Under sufficient regularity of the right-hand side, the remainder term is shown to decay rapidly with respect to powers of factors that are explicit with respect to $\eps$ and $\rho$. Applying time stepping schemes to the coupled systems yields approximations of the coefficients $z^K$ (and thus of the solution $u$ of the time-modulated acoustic wave equation), that do not suffer step size restrictions with respect to $\eps$ and are stable for step sizes $\tau$ that are independent of the period of the modulation $\eps$, if the amplitude is sufficiently small. 

We note that the approximation of the time-modulated evolution problem by a coupled system with constant coefficients has other computational advantages: Most importantly, the spatial operators have to be assembled only a single time and effective preconditioners need to be computed only once during the simulation. Other computational techniques that are particularly effective (or only available) for linear evolution problems with temporally constant coefficients include exponential integrators and time-parallel techniques.


\subsection{Outline}
The next section gives a complete problem formulation, presents the modulated Fourier expansion approach and condenses the main results of the manuscripts. A coupled time-invariant dynamical system for the coefficient functions $z^K$ of the modulated Fourier expansion is derived. Section~\ref{sect:Lap} analyzes this linear time-invariant coupled system by deriving bounds on the Laplace transforms of those coefficient functions. In the subsequent Section~\ref{sect:error bounds}, the implications of these Laplace domain bounds on the time-dependent coefficient functions are formulated and proven. Section~\ref{sect:struct-pres} shows some implications on the time-modulated acoustic wave equation, as well as a conserved property of the coefficient functions of the modulated Fourier expansion. Finally, Sections~\ref{sect:cq}--\ref{sect:numerics} discretize the coupled dynamical system, show time convergence error bounds (based on the Laplace domain bounds of Section~\ref{sect:Lap}) and illustrate the results by numerical experiments. 

\section{Problem formulation, main results and derivation}
\subsection{Problem formulation}

We consider the following acoustic wave equation with time-modulation in the system:
\begin{alignat}{2}\label{eq:modul-ac}
\begin{aligned}
  	 \partial_t^2 u(x,t) - \nabla \cdot \mu(x,t/\varepsilon)\nabla u(x,t) &= f (x,t), \quad \text{in} \quad \Omega \times [0,T],
     \\ 
     u(x,t) &= 0 \  \quad \text{on} \quad \Gamma \times [0,T],
\end{aligned}
\end{alignat}
which is completed by vanishing initial conditions, i.e. $u(\cdot,0)=\partial_t u(\cdot,0) = f(\cdot,0) = 0$. This setting is the natural description of a medium at rest that is then excited by some external source $f$, which is regular in time and space and initially vanishes. This framework is convenient to work with and simplifies the presentation. Studying the effects of nontrivial initial conditions, or non-smooth data $f$, is an interesting avenue of future research but beyond the scope of the present paper.

\subsection{Main results}
We consider two variants of the the small-amplitude modulation $\mu$. As the first model problem, we study the simplest case of the modulation
\begin{align}\label{eq:ref-mu-cos}
	\mu(t/\varepsilon) = 1 + 2\rho \cos(t/\varepsilon)
    =1+\rho(e^{it/\eps}+e^{-it/\eps}), \quad\text{where} \quad \rho,\eps\ll 1 .
\end{align}
Of key interest is the dependence of both  $\varepsilon$ and $\rho$ on the solution $u$. 
 Inserting the modulated Fourier expansion into the modulated acoustic wave equation and comparing the coefficients of the exponentials $e^{i kt/\varepsilon}$ yields the coupled system
 \begin{align}\label{eq:MFE}
 (\partial_t+ ik/\varepsilon)^2z^K_k  - \Delta z^K_k-\rho\Delta (z^K_{k-1}+z^K_{k+1}) =  \delta_{k0}f,
 \end{align}
  for all integers $k$ in the range $-K\le k \le K$ with some bounded truncation $K<\infty$. At the outer indices $k\in \{-K,K\}$, we set the terms $z^K_{K+1}$ and $z^K_{-K-1}$ to zero. For the first indices, the system reads
 \begin{align*}
 \partial_t^2z_0^K - \Delta z^K_0-\rho\Delta (z^K_{-1}+z^K_{1})  &=  f,\\
 (\partial_t+ i/\varepsilon)^2z^K_1  - \Delta z^K_1-\rho\Delta (z^K_{0}+z^K_{2}) &=  0, \\
 (\partial_t+ 2i/\varepsilon)^2z_2  - \Delta z^K_2-\rho\Delta (z^K_{1}+z^K_{3})  &=  0, \\
 &\cdots ,
 \\
 (\partial_t+ Ki/\varepsilon)^2z^K_K  - \Delta z^K_K-\rho\Delta z^K_{K-1}  &=  0.
 \end{align*}
We complete the coupled system with homogeneous Dirichlet boundary conditions and enforce vanishing initial conditions for the cofficient functions $z_k^K$ (for $-K\le k\le K$) and their temporal derivative.
We then have the following result.
\begin{theorem}[Modulated Fourier expansion]\label{thm:mfe}
Consider the solution of the system \eqref{eq:modul-ac}, formulated on a smooth domain $\Omega$, with a right-hand side $f$ that is, together with its extension by zero for $t<0$, sufficiently regular for the following right-hand sides to be finite.
Further, let the final time satisfy $$ T < \dfrac{\eps}{4\rho}.$$ Under these assumptions, the solution of \eqref{eq:modul-ac} satisfies the expansion \eqref{eq:modul-Fourier} with coefficient functions $z^K(t)$ and a remainder term $R_K(t)$, that fulfill the following properties.   The coefficient functions  $z_k^K \, \colon \,[0,T]\rightarrow H_0^1(\Omega)$ for $|k|\le K$ defined by the system  are smooth and regular, namely for any integer-valued $r>0$, we have
\begin{align*}
    \sup_{t\in [0,T]} \| \partial_t^{r+1} z^K(t)\|  \le 4T\left(\int_0^T \| \partial_t^r f(t)\|^2 \mathrm d t\right)^{\tfrac{1}{2}}.
\end{align*}
Further let $f$ have, for all times $0\le t\le T$, support in an interior domain $\Omega_0$ such that its closure is still a subset $\overline{\Omega}_0\subset \Omega $. Under these assumptions, the remainder term admits the bound
\begin{align*}
\sup_{t\in[0,T]}\|R_K(t)\|+ \|\nabla R_K (t)\|
 &\le  C \rho^{K}\eps^{K+1}  
\left(\int_0^T\left\|\partial_t^{K+1}\Delta^{K+1}f(t)\right\|^2 \, \mathrm d t\right)^{\tfrac{1}{2}},
\end{align*}
where the constant $C$ only depends on $K$ and polynomially on the final time $T$. For both of these estimates to hold, the right-hand side $f$ and sufficiently many derivatives are required to vanish initially so that the respective integrands on the right-hand side are continuous when extended by zero for $t<0$.
\end{theorem}

Secondly, we consider the more general class of small-amplitude modulations $\mu(x,t)$ of the form
\begin{align}\label{eq:more-general-mu}
\mu(x,t/\varepsilon) = \mu_0(x)+ 2\rho
\sum_{j=0}^J \widehat \mu_j (x) \cos(jt/\varepsilon)
= \mu_0(x)+ \rho
\sum_{j=-J}^J \widehat \mu_j (x) e^{ijt/\varepsilon},
\end{align}
where the real-valued coefficients $\widehat{\mu}_j$ of the dynamic component are mirrored to negative indizes (i.e. we set $\widehat{\mu}_{-j} = \widehat{\mu}_j$) and are further assumed to be sufficiently regular (i.e. $\mu_j(x)\in C^1(\Omega)$).
Moreover, the static component of $\mu$, i.e. $\mu_0$, is assumed to be positive, bounded away from zero and bounded from above, in the sense that there exist constants $c_\mu$ and $C_\mu$, such that
\begin{align}\label{eq:mu-prop}
0<c_\mu < \mathrm{ess}\, \inf_{x\in \Omega} \mu_0(x)\le  \mathrm{ess}\, \sup_{x\in \Omega} \mu_0(x) < C_\mu <\infty.
\end{align}
Moreover, we assume that the spatial coefficients of the oscillating coefficients and their derivatives are continuously bounded, in the sense that $\widehat{\mu}_j\in C^1(\Omega)$ for all $-J\le j\le J$.
The coefficient functions $z^{JK}$ are again determined by an evolution problem, which generally reads: Find, for all $0\le t\le T$, the coefficient function 
\begin{align}
z_k^{JK}(t)\in V_K = H_0^1(\Omega)^{2JK+1},
\end{align}
such that for all integer-valued $k$ in the range $-JK\le k\le JK$ it holds that
\begin{align}\label{eq:general-mfe-sys}
\left(\partial_t+ \frac{ik}{\varepsilon} \right)^2 z^K_k -\nabla \cdot \mu_0 \nabla z^K_k - \sum_{j=-J}^J\nabla \cdot \widehat{\mu_j} \nabla z^K_{k-j} = \delta_{k0}f,
\end{align}
with the spatially varying operators
\begin{align*}
A_0(x) &= -\nabla \cdot \mu_0(x) \nabla
\quad \text{and}\quad
\widehat{A}_j(x) = - \nabla \cdot \widehat{\mu_j}(x) \nabla, \quad \text{for}\quad j=-J,\dots,J.
\end{align*}
As the exact solution is assumed to vanish initially, we complete the system by setting the the coefficients $z_k^K$ and their derivatives initially to zero. 
We then have the following result.
\begin{theorem}[Modulated Fourier expansion for low spatially variable $\mu$]\label{thm:mfe-low-reg}
Consider the time-modulated acoustic wave equation with the modulation \eqref{eq:more-general-mu} which fulfills the assumption \eqref{eq:mu-prop} and further let the assumptions of Theorem~\ref{thm:mfe} hold. Then, for sufficiently small $\rho$, the following bound holds for all $r\ge 0$
\begin{align*}
    \sup_{t\in [0,T]} \| \partial_t^{r+1} z^{JK}(t)\|  \le C\left(\int_0^T \| \partial_t^r f(t)\|^2 \mathrm d t\right)^{\tfrac{1}{2}},
\end{align*}
where the constant $C$ depends on the functions $\widehat{\mu}_j$ for $|j|\le J$, the final time $T$ and on $K$, but is independent of $\rho$ and $\eps$.
The remainder term is bounded by
\begin{align*}
\sup_{t\in[0,T]}\|R_{JK}(t)\|+ \|\nabla R_{JK} (t)\|
 &\le  C \dfrac{\rho^{K}}{\eps^{K}}  
\left(\int_0^T\left\|\partial_t^{K+1}f(t)\right\|^2 \, \mathrm d t\right)^{\tfrac{1}{2}},
\end{align*}
where the constant $C$ only depends on $K$, polynomially on the final time $T$ and the modulation $\mu$ through its coefficients $\widehat{\mu}_j$ for $|j|\le J$, but is independent of $\rho$ and $\eps$. For both of these estimates to hold, the right-hand side $f$ and sufficiently many derivatives are required to vanish initially so that the respective integrands on the right-hand side are continuous when extended by zero for $t<0$.
\end{theorem}
In the case of spatially variable modulations $\mu$, covered by Theorem~\ref{thm:mfe-low-reg}, stronger decay conditions can probably be derived, but overcoming the technical complications encountered by the spatial dependency of $\mu$ is beyond the scope of this paper (some aspects of the spatial modulation is formulated in Section~\ref{sect:spatially-variable}). 
For the sake of presentation and to demonstrate the technique that leads to the rapid decay condition of Theorem~\ref{thm:mfe}, we focus on the formulation \eqref{eq:MFE} in the following sections. 

\begin{remark}[Rapidly oscillating excitations]
The coupled dynamical systems \eqref{eq:MFE}/\eqref{eq:general-mfe-sys} can further be formulated for source terms that themselves have the form of a modulated Fourier expansions, i.e.
 \begin{align*}
     f(x,t/\eps) = \sum_{\ell=-L}^L f^L_\ell(x,t) e^{i\ell t/\eps},
 \end{align*}
with slowly varying coefficients $f_\ell^L$. Then, the corresponding coupled system \eqref{eq:MFE}/\eqref{eq:general-mfe-sys} can be formulated only with the slowly varying coefficients $f_\ell^L$ on the right-hand side. Under sufficient decay conditions of $f_\ell^L$, the decay conditions of Theorems~\ref{thm:mfe}--\ref{thm:mfe-low-reg} can be carried over, or slightly adapted where the decay starts at the first index at which $f_\ell^L$ vanishes.
\end{remark}

\begin{remark}[Limitations of the approach]
The results of this manuscript hold for small $\varepsilon$, as long as the ratio $\rho t/\eps$ is bounded by a constant that may exponentially appear as a multiplicative factor on the right-hand side of all error bounds. This limitation, which limits the scope of the work on timescales $t\sim 1$ to small-amplitude modulations, is inherited of the time-modulated acoustic wave equation, which itself is only known to be stable up to such time scales (see Proposition~\ref{prop:energy-estimate}). The coupled dynamical system is stable only on such time scales as well, which is the consequence of the reduced domain of analyticity of the Laplace domain resolvent of Proposition~\ref{prop:well-posedness} (corresponding to the condition that $\Re s > c  \rho/\eps$ is required for the system to be well-posed). A stronger theory, if at all possible, would need to overcome both of these difficulties, to extend the theory present for longer times. Such results would be interesting additions to the literature, but are beyond the scope of this manuscript.
\end{remark}
\begin{remark}[Symmetry of the coefficients]\label{rem:symmetry}
We remark that conjugating the coupled system reads, for real-valued $f$,
\begin{align}\label{eq:td-coupled-system-conj}
 (\partial_t -ik/\varepsilon)^2\overline{z}_k  - \Delta \overline{z}_k-\rho\Delta (\overline{z}_{k-1}+\overline{z}_{k+1}) =  \delta_{k0} f.
 \end{align}
A comparison with the original system and the uniqueness of the solution (which will be established subsequently in the manuscript) shows that 
 \begin{align}\label{eq:td-symmetry-zk}
     \overline{z}_k = z_{-k},
 \end{align}
 which shows that the modulated Fourier expansion is real-valued, as long as $f$ has no complex-valued components.
\end{remark}
\section{Laplace domain analysis of the coupled system}\label{sect:Lap}
Applying the Laplace transformation on both sides of the system \eqref{eq:MFE} gives a formulation for the Laplace transforms of the coefficient functions $z_k^K$, which are denoted in the following by $\widehat{z}_k^K$. For the Laplace parameter $s\in \mathbb C$, whose real part is assumed to be positive (i.e. $\Re s>0$), we enforce for all $-K\le k\le K$
\begin{align}\label{eq:mfe-sys-lap}
    \left(s+ \dfrac{ik}{\varepsilon} \right)^2 \widehat{z}^K_k -\Delta ( \widehat{z}^K_k + \rho \widehat{z}^K_{k-1}+\rho \widehat{z}^K_{k+1} )= \delta_{k0}\widehat{f}.
\end{align}
Here, we used that the coefficient functions $z^K_k$ vanish initially. More general initial conditions would include additional terms on the right-hand side, which is neglected in the manuscript to simplify the presentation, but is not a restriction of the approach overall. In the following, the functional analytic setting of this system is introduced to give a rigorous formulation of the Laplace domain coupled system \eqref{eq:mfe-sys-lap}.
\begin{remark}[Connection to coupled harmonics]
There has been significant recent effort to understand the system of coupled harmonics that arises from discretizing truncating quasi-periodic temporal boundary conditions 
by a temporal Fourier base and limiting the range of temporal frequencies (see, e.g. \cite{NHA24,HD24,HR25,ACHR24}). The resulting approximations to Floquet exponents are precisely the poles of the resolvent of \eqref{eq:mfe-sys-lap} (analogous to resonance frequencies of passive media, which are poles of the resolvent $R(s)$ of the Helmholtz problem).
\end{remark}
\noindent In the following, we install the coupled system on the appropriate space for the Laplace transforms of the coefficient functions, namely 
\begin{align*}
V_K = H_0^1(\Omega)^{2K+1}.
\end{align*}
Acting on this space are truncated operators corresponding to the differential operators appearing in the time-modulated acoustic wave equation, which are denoted by
\begin{align*}
D_K (s)\, \colon \, V_K \rightarrow V_K', 
\quad \text{and}\quad
\mathcal{T}_\Delta \, \colon \, V_K \rightarrow V_K'.
\end{align*}
Their action on the components of the coefficient functions reads 

\begin{align*}
 D_K (s) z = \left((s+\tfrac{ik}{\eps} )z_k\right)_{|k|\le K}
 \quad \text{and} \quad 
\mathcal{T}_\Delta z = \left(-\rho \Delta z_{k-1} -\Delta z_k
 -\rho \Delta z_{k+1} \right)_{|k|\le K}.
\end{align*}
At the indices $k=K,-K$ we use the convention that $z_{K+1}=z_{-K-1}=0$ to simplify the notation. The right-hand side is denoted by $$\widehat f^K = \left(\delta_{k0} \widehat f \right)_{-K\le k \le K}.$$
The strong form of the Laplace domain system of coupled equations then reads
\begin{align}\label{eq:trunc-sys-th}
 (D_K^2 (s) + \mathcal{T}_\Delta ) \widehat z^K 
 &=  \widehat f^K,
\end{align}
and admits the weak formulation 
\begin{align}\label{eq:weak-K}
a_K(v,\widehat z^K) \coloneqq \left\langle v, 
 (D_K^2 (s) + \mathcal{T}_\Delta ) \widehat z^K \right \rangle
 &= 
 \left\langle v, \widehat f^K
\right \rangle
\quad \text{for all} \quad v \in V_K.
\end{align}
The anti-duality $\left \langle \cdot,\cdot \right\rangle$ of $V_K\times V_K'$ is the sum of each component 
\begin{align*}
\left \langle v,w\right\rangle
= 
\sum_{k=-K}^K \left \langle v_k,w_k\right\rangle_{H^{1}_0(\Omega)\times H^{-1}(\Omega)}\quad \text{for all} \quad v,w \in V_K,
\end{align*}
where each anti-duality $\left \langle \cdot ,\cdot\right\rangle_{H^{1}_0(\Omega)\times H^{-1}(\Omega)}$ associated to a single component is an extension of the $L^2(\Omega)$ pairing. To simplfiy the notation for the pairing associated with the $H^1(\Omega)$ semi-norm, we write
\begin{align*}
\left\langle z_1,z_2 \right\rangle_1
=
\left\langle \nabla z_1,\nabla z_2 \right\rangle .
\end{align*}
The following section analyzes gives sufficient conditions such that the weak formulation \eqref{eq:weak-K}  admits a unique solution, that is bounded by the right-hand side $\widehat f^K$. In view of time-dependent bounds, particular emphasis is made to track the dependence of $s$ on the bounds.
\subsection{Well-posedness of the coupled system}
For $s$ with sufficiently large real part, the coupled system \eqref{eq:weak-K} is well-posed, which is investigated in this subsection. More precisely, the statements of this section hold for Laplace parameters that fulfill
\begin{align}\label{assumpt:well-posedness}
    \Re s > \dfrac{4\rho}{\varepsilon}.
\end{align}
This assumption ensures that the mixed terms that appear in \eqref{eq:weak-K} can be absorbed, a fact that is elucidated in the following Lemma.
\begin{lemma}[Coercivity of $a_K$]\label{lem:coercive-K}
Let $s \in \mathbb C$ fulfill the assumption \eqref{assumpt:well-posedness}. Under this assumption, we have the coercivity estimate
\begin{align*}
\Re \, a_K(D_K (s) z,  z)
 \ge \tfrac{1}{2}\,\Re s\,\left(\| D_K (s)z\|^2
 +\|\nabla  z \|^2\right),
\end{align*}
which holds for all $z\in V_K$.
\end{lemma}
\begin{proof}
A direct consequence is the coercivity of the first summand via
\begin{align*}
\Re \left\langle D_K (s) z, 
 D_K^2 (s) z \right \rangle   
 &= \Re s \|D_K (s)  z \|^2.
\end{align*}
We start to rewrite the second summand, which reads
\begin{align*}
&\Re \left\langle D_K (s) z , 
\mathcal{T}_\Delta  z \right \rangle
 =\Re
 \sum_{k=-K}^K \left(\overline{s}-\dfrac{ik}{\eps} \right)
 \left\langle z_k, 
 \rho z_{k-1}+ \widehat{z}_k+\rho z_{k+1} \right \rangle_1
 \\ &
 =\Re s  \sum_{k=-K}^K 
  \|z_k\|_1^2+
  \rho \Re  \left(\overline{s}
  \sum_{k=-K}^K
 \left\langle z_k, 
 z_{k-1}+z_{k+1} \right \rangle_1\right)
 -  \dfrac{\rho}{\eps} \sum_{k=-K}^K\Re ik
 \left\langle z_k, 
 z_{k-1}+z_{k+1} \right \rangle_1
 \\ &= 
 \quad\quad\quad\quad\mathrm{I}
 \quad\quad \, \, + 
 \quad\quad\quad\quad\quad\quad\mathrm{II} 
 \quad\quad\quad\quad \quad\quad\quad\ \, \,\,-
 \quad\quad\quad\quad\mathrm{III}.
\end{align*}
Our intention is to absorb the terms in the second and third series II and  III with the terms appearing in the first sum I. The series appearing as a factor in the second summand is crucially real-valued, due to the identity
\begin{align*}
 & \sum_{k=-K}^K
 \left\langle z_k, 
z_{k-1}+z_{k+1} \right \rangle_1
 = 2 \Re  
  \sum_{k=-K}^K
 \left\langle z_k, z_{k+1} \right \rangle_1 
.
\end{align*}

The second summand can therefore be absorbed for $\rho$ small enough, since 
\begin{align*}
 \mathrm{II} = \rho \Re  \overline{s}
  \sum_{k=-K}^K
 \left\langle z_k, 
 z_{k-1}+z_{k+1} \right \rangle_1 
&= 
2\rho \,\Re s \,\,\Re  
  \sum_{k=-K}^K
 \left\langle z_k, z_{k+1} \right \rangle_1 
 \\ &\le 
 \rho 2\Re s \sum_{k=-K}^K \left \| z_k \right\|^2.
\end{align*}
The final term is estimated in a similar way and used to replace the coefficient $k$ of the summands by a bounded coefficient, by a telescope-like argument. We rewrite the series by analyzing two summands at once, which gives
\begin{align*}
 \dfrac{\varepsilon}{\rho} \mathrm{III}
 &= \Im (-K)
 \left\langle z_{-K}, 
z_{-K+1} \right \rangle_1
 +\Im(-K+1)
 \left\langle z_{-K+1}, 
 z_{-K}+z_{-K+2} \right \rangle_1+\ldots
\\  &=
 \Im
 \left\langle z_{-K+1}, 
z_{-K} \right \rangle_1
 +\Im(-K+1)
 \left\langle z_{-K+1}, z_{-K+2} \right \rangle_1 
 +\ldots,
 \\ &\ldots 
 \\ & = 
\Im
 \left\langle z_{-K+1}, 
z_{-K} \right \rangle_1+
\ldots 
+ \Im
 \left\langle z_{K}, 
z_{K-1} \right \rangle_1. 
\end{align*}
The key argument here is that each summand in this series "almost" eliminates its successor, up to the coefficient $1$.
Applying the Cauchy--Schwarz inequality and Young's inequality on each summand of the right-hand side therefore gives the bound
\begin{align*}
 \mathrm{III}  &\le  
 \dfrac{\rho}{\varepsilon}\sum_{k=-K}^{K} 
 \left\| z_{k} \right\|_1^2.
 \end{align*}
Under the stated assumption of $4\rho < \varepsilon  \Re s, $ we can therefore absorb the term with the bound from below for I.

\end{proof}
The coercivitiy result for $a_K$ is the key for the well-posedness result and the bound of the coefficients $\widehat z^K$. The continuity of the bilinear form is obtained by applying the Cauchy--Schwarz inequality to both arising terms, which is formulated in the following. 
\begin{lemma}\label{lem:aK-bound}
The bilinear form $a_K$ is bounded, namely it holds that
\begin{align*}
|a_K(D_K (s)  v, w) |
&\le 
C\left(|s|+\dfrac{K}{\eps}\right) 
\left( \left\|D_K(s)  v\right\| 
\left\|D_K(s)  w \right\|
+
 \left\|\nabla  v\right\|\left\|\nabla  w \right\|\right)´,
\end{align*}
where $C$ only depends on an upper bound of $\rho$.
\end{lemma}

The well-posedness of the truncated Laplace domain system \eqref{eq:weak-K} now follows from standard Lax--Milgram theory, under the stated bound \eqref{assumpt:well-posedness} from below on the real part of the Laplace parameters $s$.
\begin{proposition}\label{prop:well-posedness}
For any $s \in \mathbb C$ that fulfills the assumption \eqref{assumpt:well-posedness},
the system \eqref{eq:weak-K} has the unique solution $\widehat z^K$, which is bounded by
\begin{align*}
\| D_K (s)\widehat z^K\|^2
 +\|\nabla \widehat z^K \|^2
 \le \dfrac{ 4}{(\Re s)^2} \|\widehat f \|^2.
\end{align*}
\end{proposition}
\begin{proof}
The bound and the  well-posedness is the direct consequence of the Lax--Milgram Theorem with the coercivity of Lemma~\eqref{lem:coercive-K} and the bound of Lemma~\eqref{lem:aK-bound}. The bound on the solution $\widehat z^K$ is seen by the standard chain of inequalities
\begin{align*}
\| D_K (s)\widehat z^K\|^2
 +\|\nabla \widehat z^K \|^2
 &\le
 \dfrac{2}{\Re s} \Re a_K(D_K(s) \widehat z^K,\widehat z^K)
 =  \dfrac{2}{\Re s}\left(D_K(s)\widehat z^K,\widehat f^K \right)
 \le 
 \tfrac{1}{2}\| D_K (s)\widehat z^K\|^2 
 + \dfrac{2}{(\Re s)^2}\|\widehat f\|^2,
\end{align*}
by absorbing the first summand on the right-hand side.
\end{proof}

The goal of this section is to show error bounds for the modulated Fourier expansions that are explicit with respect to a high power of the small parameters $\rho$ and $\varepsilon$. These bounds are the consequence of the decay of the coefficient functions $\widehat{z}_k^K$, which mitigate the effect of the truncation at the index $K$. The following section proves such decay estimates.

\subsection{Decay of the coefficients $\widehat{z}^K_k$}
To bound the truncation error, we use a decay property of the coefficients $\widehat z^K_k$ with respect to the index $k$. 
Rearranging the $k$-th equation in the coupled system \eqref{eq:mfe-sys-lap} reads (for $k\neq 0$)
\begin{align}\label{eq:mfe-sys-lap-rearranged}
    \left(s+\dfrac{ik}{\eps}\right)^2 \widehat z^K_k - \Delta \widehat z^K_k =\rho \Delta \widehat z^K_{k-1}+\rho \Delta \widehat z^K_{k+1}.
\end{align}
Resolvent bounds for the operators associated to the standard Helmholtz problem (with the wave numbers $s+ik/\eps$) therefore give preliminary bounds on the coefficients. The following paragraphs collect such estimates.
\subsubsection*{Preliminary: Resolvent bounds for the unmodulated operators}
 The resolvent of the standard Helmholtz equation is denoted by
\begin{align}\label{eq:def-res}
R(s) = (s^2 - \Delta )^{-1}.
\end{align}
Note that two resolvents on the appropriate compatible spaces with different arguments $s_1,s_2 \in\mathbb C$ (with positive real part) commute, since we have  
\begin{align*}
    R(s_1)R(s_2) =  ((s_2^2 - \Delta ) (s_1^2 - \Delta ) )^{-1} = ((s_1^2 - \Delta ) (s_2^2 - \Delta ) )^{-1} 
    = R(s_2)R(s_1).
\end{align*}
Moreover, the same argument gives the commutation of any Helmholtz resolvent with the Laplace operator, since we have (on compatible spaces)
\begin{align*}
\Delta R(s) &= (s^2 \Delta^{-1} - \Delta \Delta^{-1} )^{-1}
=( \Delta^{-1}s^2 -\Delta^{-1} \Delta  )^{-1} = R(s)\Delta .
\end{align*}
The Helmholtz resolvents satisfy the following bounds.
\begin{lemma}\label{lem:th-res-bounds}
	For arbitrary $\Re s>0$, we have the following bounds on the resolvents of the Helmholtz problem
	 \begin{align}
	 \left\| R(s)\right\|_{L^2(\Omega)\leftarrow L^2(\Omega)}&\le \dfrac{1}{|s|\,\Re s},\\
	 \left\| R(s)\right\|_{H^1_0(\Omega)\leftarrow L^2(\Omega)}&\le \dfrac{1}{\sqrt{2}\,\Re s}.
    \end{align}
\end{lemma}
\begin{proof} Bounds of this type are well-known in the context of Laplace domain techniques for the acoustic wave equation (find e.g. \cite[Section~4.4]{BS22} for a collection of estimates). For the convenience of the reader, we briefly prove the stated estimates. Let $\widehat{w}$ defined by $ (s^2-\Delta)\widehat{w}=   \widehat{g} $. Dividing both sides by $s$, testing that equation with $\widehat{w}$ and taking the real part gives
\begin{align*}
&\Re s \int_\Omega |\widehat{w} |^2 + |s^{-1}\nabla \widehat{w}|^2 \mathrm d x
= \Re \int_\Omega \,\overline{\widehat{w}} \cdot s^{-1} \widehat{g} \, \mathrm d x ,
\\&\le \int_\Omega \left|  \widehat{w}\right| \left|s^{-1}  \widehat{g} \right| \, \mathrm d x 
\le \int_\Omega \dfrac{\Re s}{2} \left|  \widehat{w}\right|^2+ \dfrac{1}{2\,\Re s}\left| s^{-1} \widehat{g} \right|^2 \mathrm d x
.
\end{align*}
Absorbing the first summand on the right and neglecting the second summand on the left-hand side gives the first result. In turn, neglecting the first summand on the left-hand side gives the second bound.
\end{proof}
We often encounter the combination of a Helmholtz resolvent with a Laplace operator (due to the occurrence of their combination in \eqref{eq:mfe-sys-lap-rearranged}). The following lemma gives several operator bounds for this combination.
\begin{lemma}\label{lem:th-res-bounds-R-Delta}
	For arbitrary $\Re s>0$, we have the following operator bounds
	 \begin{align}
	\left\| R(s)\Delta\right\|_{H_0^1(\Omega)\leftarrow H_0^1(\Omega)}&\le \dfrac{|s|}{\Re s},
\\
	 \left\| R(s)\Delta\right\|_{L^2(\Omega)\leftarrow H_0^1(\Omega)}&\le \dfrac{1}{\sqrt{2}\Re s},
\\
	\left\| R(s)\Delta\right\|_{L^2(\Omega)\leftarrow L^2(\Omega)}&\le \dfrac{|s|}{\Re s}.
    \end{align}
\end{lemma}
\begin{proof} 
Let $\widehat{w}$ be defined by $ (s^2-\Delta)\widehat{w}=  \Delta \widehat{g} $. Then, we obtain the bound
\begin{align*}
&\Re s \int_\Omega   |\widehat{w} |^2 + |s^{-1}\nabla \widehat{w}|^2  \mathrm d x
= \Re \int_\Omega \,s^{-1}\nabla \overline{\widehat{w}} \cdot \nabla \widehat{g} \, \mathrm d x ,
\\&\le \int_\Omega \left| s^{-1} \nabla \widehat{w}\right| \left|\nabla \widehat{g} \right| \, \mathrm d x 
\le \int_\Omega \dfrac{\Re s}{2} \left| s^{-1} \nabla \widehat{w}\right|^2+ \dfrac{1}{2\,\Re s}\left|\nabla \widehat{g} \right|^2 \mathrm d x
.
\end{align*}
Absorbing the first summand on the right-hand side into the left-hand side then yields
\begin{align*}
& \int_\Omega   |\widehat{w} |^2 + \tfrac12 |s^{-1}\nabla \widehat{w}|^2 \mathrm d x
\le
\dfrac{1}{2\,(\Re s)^2}\int_\Omega \left|\nabla \widehat{g} \right|^2 \, \mathrm d x
,
\end{align*}
and consequently, for $\Re s > 0$, 
\begin{align}\label{eq:bound-rec}
\left\| \widehat{w} \right\|_{} &\le \dfrac{1}{\sqrt2\,\Re s} \left\| \widehat{g} \right\|_{1} ,
\\
\left\| \widehat{w} \right\|_{1} 
&\le \dfrac{|s|}{\Re s} \left\| \widehat{g} \right\|_{1},
\end{align}
which proves the first two bounds. For the final bound, we use the spectral theorem for the Laplace operator to obtain
\begin{align*}
\| R(s) \Delta \|_{L^2(\Omega)\leftarrow L^2(\Omega)}&= \sup_{\lambda >0}  \dfrac{\lambda}{| s^2+\lambda|} 
= \sup_{\lambda >0}  \dfrac{|s|\lambda}{ | |s|^2s+\overline{s} \lambda |} 
\le \sup_{\lambda >0}  \dfrac{|s|\lambda}{ \Re s ( |s|^2 +\lambda )} \le
\dfrac{|s|}{\Re s}.
\end{align*}
\end{proof}



\subsection*{Decay estimates for $\|\widehat z_k^K \|_1$}
The resolvent estimates yields decay estimates by rearranging \eqref{eq:mfe-sys-lap-rearranged} and applying the bound of Lemma~\ref{lem:th-res-bounds}. For any $k$ within the appropriate range, this yields
\begin{align*}
    \| \widehat z^K_k \|_{1} &=\rho  \left \|  \left( (s+ik/\eps )^2 - \Delta \right)^{-1}\Delta (\widehat z^K_{k-1}+\widehat z^K_{k+1}) \right\|_{1}
 \le 
    \left( \dfrac{\rho}{\eps} \dfrac{ |k| }{\Re s}+ \rho \dfrac{ |s|}{\Re s} \right)  \left( \big \| \widehat z^K_{k-1} \big\|_{1}
    +\big \| \widehat z^K_{k+1} \big\|_{1}\right).
\end{align*}

We start with the fact that the coefficients are bounded due to Proposition~\ref{prop:well-posedness}, which gives under the stated assumptions, for any $k$ with $-K\le k\le K$, the estimate
\begin{align}\label{eq:basic-bound}
\left\|\widehat z_{k}^K \right\|_1
\le \left(
\sum_{k'=-K}^K \left \| \widehat z_{k'}^K\right\|_1^2 \right)^{\tfrac{1}{2}}
\le \dfrac{2}{\Re s} \left\| \widehat f\right\|.
\end{align}
For any \(k\neq 0\), however, we can first use \eqref{eq:bound-rec} and then use this bound, which gives 
\begin{align*}
\|\widehat z^K_k\|_1
&\le    \left( \dfrac{\rho}{\eps} \dfrac{|k| }{\Re s}+ \rho \dfrac{|s|}{\Re s} \right)  \left( \big \| \widehat z^K_{k-1} \big\|_{1}
    +\big \| \widehat z^K_{k+1} \big\|_{1}\right)
\le  \left( \dfrac{\rho}{\eps}\dfrac{ 2K}{\Re s}+ \rho \dfrac{2|s|}{\Re s} \right) \dfrac{2}{\Re s} 
 \left\| \widehat f\right\|.
\end{align*}
In comparison with the bound \eqref{eq:basic-bound}, we have gained an additional factor that tends linearly to zero as $\rho/\eps \rightarrow 0$ (for fixed $s$). An induction over this process, namely using the bound for any $|k|\ge m$ for $m=0,\dots,K$, then gives the following bound.
\begin{proposition}\label{prop:decay-no-spatial-reg}
Let $s \in \mathbb C$ satisfy the assumption \eqref{assumpt:well-posedness}.
For any $k$ in the appropriate range $-K\le k \le K$, the coefficients $\widehat z^K_k$ decay with respect to $k$ in the sense that
\begin{align*}
\|\widehat z^K_k\|_1 
 &\le  \left( \dfrac{\rho}{\eps}\dfrac{ 2 K}{\Re s}+ \rho \dfrac{2 |s|}{\Re s} \right)^{k} \dfrac{2}{\Re s} 
 \left\| \widehat f\right\|.
\end{align*}
The analogous bound holds for $|k|$ for negative indices $k<0$.
\end{proposition}
\begin{proof}
    As discussed in the paragraphs leading up to the proposition, the statement follows by induction over $|k|$, by the direct combination of the resolvent bounds of Lemma~\ref{lem:th-res-bounds} with the standard bound on the coefficients described in Proposition~\ref{prop:well-posedness}.
\end{proof}
In the previous bound, no assumptions were made on the spatial regularity of the right-hand side $\widehat f$. Under additional spatial regularity of the right-hand side and the domain $\Omega$, higher rates of decay hold.

\begin{proposition}\label{prop:decay-spatial-reg}
Let $\Omega$ be smooth and $s \in \mathbb C$ have sufficiently large real part to fulfill assumption \eqref{assumpt:well-posedness}. Moreover, let $\widehat f$ be supported on $\Omega_0$, an arbitrary open subset of $\Omega$ such that $\overline{\Omega_0} \subset \Omega$. 
For any $k$ in the appropriate range $0\le k \le K$, the coefficients $\widehat z^K_k$ decay with respect to $k$ in the sense that
\begin{align*}
   \|\widehat z^K_k \|_1
\le  
  C \rho^k\eps^{k}   \dfrac{|s|^k }{(\Re s)^{2k+1} } \left\|\Delta^{k+1} \widehat f\right\|,
\end{align*}
where the constant $C$ depends only on $K$. For negative indices $k<0$, the analogous bound holds with the corresponding exponents $|k|$.
\end{proposition}
\begin{proof}
Throughout the proof, we formally use powers of $\Delta$ and its inverse, a process that is made rigorous by \cite[Theorem 11.2]{HW96}. We note that due to the commutation of the Laplace operator with the resolvent of the coupled system, the bound \eqref{eq:basic-bound} extends to powers of the Laplacian in the sense that we have, for arbitrary $\ell\in\mathbb N$ (and sufficiently smooth domains $\Omega$), the estimate
\begin{align}\label{eq:basic-bound-lap}
\left\|\Delta^\ell\widehat z_{k}^K \right\|_1
\le \left(
\sum_{k'=-K}^K \left \| \Delta^\ell\widehat z_{k'}^K\right\|_1^2 \right)^{\tfrac{1}{2}}
\le \dfrac{2}{\Re s} \left\| \Delta^\ell \widehat f\right\|,
\end{align}
if $\widehat f$ is sufficiently regular for the right-hand side to be finite. Note that the assumption that $\widehat f$ is supported in the interior of $\Omega$ ensures compatibility of the boundary values.
Rewriting \eqref{eq:mfe-sys-lap-rearranged} in terms of the resolvent and applying the Laplace operator $\Delta$ on both sides then gives
\begin{align*}
  &\Delta \widehat z^K_k =\rho  \Delta R(s+\tfrac{ik}{\eps}) \Delta (\widehat z^K_{k-1}+\widehat z^K_{k+1}),
\end{align*}
and, for $k\neq 0$, the estimate
\begin{align*}
   \|\widehat z^K_k \|_1 
  &\le \rho  \|\Delta^{-1} \|_{H^1_0(\Omega) \leftarrow L^2(\Omega)} \| R(s+\tfrac{ik}{\eps}) \|_{L^2(\Omega)\leftarrow L^2(\Omega)}\left\|\Delta^2 (\widehat z^K_{k-1}+\widehat z^K_{k+1})\right\|
 \le
  C \dfrac{2\rho}{(\Re s)^2 |s+\tfrac{ik}{\eps}|}  \left\|\Delta^2 \widehat f\right\|.
\end{align*}
We note that for arbitrary $|k|\le K$, we have the elementary estimate 
\begin{align*}
    \dfrac{1 }{|s+ik/\eps|} &=  
 \dfrac{|s|}{|s||s+ik/\eps|}\le  \dfrac{|s|}{\Re s \max\left(|s|,\big|s+\tfrac{ik}{\eps} \big|\right)} 
\le \dfrac{2 \eps |s| }{|k|\Re s } .
\end{align*}
Inserting this estimate into the previous bound gives
\begin{align*}
   \|\widehat z^K_k \|_1 
& \le
  \eps\rho   \dfrac{C|s| }{(\Re s)^3 } \left\|\Delta^2 \widehat f\right\|,
\end{align*}
where $C$ depends only on $k$. More generally, for arbitrary integer-valued $\ell\ge 1$, we have for $k \neq 0$
\begin{align*}
   \|\Delta^\ell \widehat z^K_k \|
  &\le \rho   \| R(s+\tfrac{ik}{\eps}) \|_{L^2(\Omega)\leftarrow L^2(\Omega)}\left\|\Delta^{\ell+1} (\widehat z^K_{k-1}+\widehat z^K_{k+1})\right\|
\\ &\le  \rho\eps  \dfrac{2|s| }{|k|(\Re s)^2 }  \left\|\Delta^{\ell+1} (\widehat z^K_{k-1}+\widehat z^K_{k+1})\right\| 
  \le \rho\eps   \dfrac{C|s| }{(\Re s)^3 } \left\|\Delta^{\ell+1} \widehat f\right\|.
\end{align*}
The constant $C$ again only depends on $k$. Using the same induction as in the previous Proposition~\ref{prop:decay-no-spatial-reg} then gives the result.
\end{proof}


\begin{remark}[The case $K=1$]
When only a single coupled harmonic is considered, the Laplace domain coupled system is given by the block form
\begin{align}
D^2_K(s)+\mathcal{T}_\Delta =
\begin{pmatrix}
 (s-\tfrac{i}{\eps})^2-\Delta  & -\rho \Delta & 0 
\\
-\rho \Delta & s^2-\Delta  & -\rho \Delta 
\\
0 & -\rho \Delta & (s+\tfrac{i}{\eps})^2-\Delta  
\end{pmatrix}.
\end{align}
The inverse of this operator can be written as a Neumann series, which we formally write as
\begin{align*}
(D^2_K(s)´+\mathcal{T}_\Delta)^{-1} =\sum_{n=0}^\infty (\rho \Delta )^n
\begin{pmatrix}
0 & R(s-\tfrac{i}{\eps}) & 0 
\\
R(s) & 0 & R(s) 
\\
0 & R(s+\tfrac{i}{\eps})& 0.
\end{pmatrix}^n
\begin{pmatrix}
 R(s-\tfrac{ik}{\eps})  & 0 & 0 
\\
0 & R(s)  & 0
\\
0 & 0 & R(s+\tfrac{i}{\eps})
\end{pmatrix}.
\end{align*}    
We note that the second factor of the second power of the iterative matrix has the form
\begin{align*}
&\begin{pmatrix}
0 & R(s-\tfrac{ik}{\eps}) & 0 
\\
R(s)& 0 & R(s)
\\
0 & R(s+\tfrac{i}{\eps}) & 0.
\end{pmatrix}^2
 =  R(s)
 \begin{pmatrix}
R(s-\tfrac{ik}{\eps}) & 0& R(s-\tfrac{i}{\eps}) 
\\
0 &    R(s-\tfrac{ik}{\eps})+R(s+\tfrac{i}{\eps})  & 0
\\
R(s+\tfrac{ik}{\eps})  & 0 & R(s+\tfrac{i}{\eps}) 
\end{pmatrix}.
\end{align*}
Finally, the third power of the second factor reads 
\begin{align*}
&\begin{pmatrix}
0 & R(s-\tfrac{ik}{\eps}) & 0 
\\
R(s) & 0 & R(s)
\\
0 & R(s+\tfrac{i}{\eps}) & 0.
\end{pmatrix}^3
= R(s) (R(s-\tfrac{i}{\eps})+R(s+\tfrac{i}{\eps})) 
\begin{pmatrix}
0 & R(s-\tfrac{ik}{\eps}) & 0 
\\
R(s)& 0 & R(s)
\\
0 & R(s+\tfrac{i}{\eps}) & 0
\end{pmatrix}.
\end{align*}
With the commutation of the resolvent and the Laplace operator we therefore obtain explicit formulas, up to high-orders of $\rho$ and $\eps$, for the solution of the coupled system: For the leading order component with index zero we have
\begin{align*}
\widehat{z}_{0}^1
     &=  R(s) \widehat f+ 
     \mathcal O (\rho^2 \eps) ,
\end{align*}
and for the two corrections we obtain
\begin{align*}
    \widehat{z}_{-1}^1
     &= \rho R(s-\tfrac{i}{\eps}) \Delta R(s) \widehat f+
    \mathcal O (\rho^3\eps^2),
    \\ 
        \widehat{z}_{1}^1 &= \rho R(s+\tfrac{i}{\eps}) \Delta   R(s)\widehat f+
    \mathcal O (\rho^3\eps^2),
\end{align*}
for sources $\widehat f$ that satisfy that $\|\Delta^2\widehat f\|$ decays sufficiently rapidly as $|s|\rightarrow \infty$ (which is the case if it corresponds to a Laplace transform of a time-dependent function $\Delta^2 f$ ). 


\end{remark}
\subsection{Spatially variable general modulations}\label{sect:spatially-variable}
This section briefly discusses the implications of more general modulations $\mu$ of the type \eqref{eq:more-general-mu}, which are the sum of several varying systems, whose coupled Laplace domain formulation reads as follows: find the coefficients $\widehat{z}_k^{JK}\in H^1_0(\Omega)^{2JK+1}$, such that for all $k$ in the range $-JK\le k \le JK$, it holds that
\begin{align*}
    \left(s+\dfrac{ik}{\eps}\right)^2 \widehat z^{JK}_k + A_0 \widehat z^{JK}_k +\rho 
    \sum_{j=-JK}^{JK}\widehat{A}_{j}  \widehat z^K_{k-j} =\delta_{k0}\widehat{f}.
\end{align*}
The system is naturally formulated on the space
\begin{align*}
V_{JK} = H_0^1(\Omega)^{2JK+1},
\end{align*}
on which we have the operators
\begin{align*}
D_{JK}(s) \, \colon \, V_{JK} \rightarrow V_{JK}', 
\quad \text{and}\quad
\mathcal{T}_A \, \colon \, V_{JK} \rightarrow V_{JK}', 
\end{align*}
whose action on a coefficient function $z\in H_0^{JK}(\Omega)$ is determined by 
\begin{align*}
 D_{JK} (s) z = \left((s+\tfrac{ik}{\eps} )z_k\right)_{|k|\le JK}
 \,\, \text{and} \,\,
\mathcal{T}_A z = \left(A_0+\rho (\widehat{A}_{-J} z_{k-J} +\dots+
 \widehat{A}_{J}z_{k+J}) \right)_{|k|\le JK}.
\end{align*}
With these operators, we rewrite the Laplace domain system of coupled equations
\begin{align}\label{eq:trunc-sys-th}
 (D_{JK}^2 (s) + \mathcal{T}_A) \widehat z^{JK} 
 &=  \widehat f^{JK},
\end{align}
or weakly formulated as 
\begin{align}\label{eq:weak-JK}
a_{JK}(v,\widehat z^{JK}) \coloneqq \left\langle v, 
 (D_{JK}^2 (s) + \mathcal{T}_A ) \widehat z^{JK} \right \rangle
 &= 
 \left\langle v, \widehat f^{JK}
\right \rangle
\quad \text{for all} \quad v \in V_{JK}.
\end{align}

\begin{proposition}\label{prop:well-posedness-spatially-variable}
Let the parameter $\mu$ fulfill the properties described in \eqref{eq:mu-prop}.
There is a constant $c>0$, such that for all  $s \in \mathbb C$ with $\Re s > c\rho/\eps$ and for sufficiently small $\rho$, we have the following result. For all $z\in V_{JK}$ it holds that
\begin{align*}
\Re \, a_{JK}(D_{JK} (s) z,  z)
 \ge 
 C_0\,\Re s\,\left(\| D_{JK} (s)z\|^2
 +\|\nabla  z \|^2\right).
\end{align*}
Moreover, the system \eqref{eq:weak-JK} has the unique solution $\widehat z^{JK}$, which is bounded by
\begin{align*}
\| D_{JK} (s)\widehat z^{JK}\|^2
 +\|\nabla \widehat z^{JK} \|^2
 \le \dfrac{C_1}{(\Re s)^2} \|\widehat f \|^2.
\end{align*}
The constants $C_0$ and $C_1$ only depend on the modulation $\widehat \mu$.
\end{proposition}
\begin{proof}
The coercivity of $a_{JK}$ is derived by repeating the proof of Lemma~\ref{lem:coercive-K} for each summand and choosing $c$ sufficiently small to absorb the mixed terms. The well-posedness then follows by applying the Lax--Milgram Lemma.
\end{proof}

\begin{proposition}\label{prop:decay-spatial-reg-JK}
Let $\Omega$ be smooth, let the coefficients $\mu_0$ fulfill the property \eqref{eq:mu-prop} and further let $\hat{\mu}_j\in C^{2k+1}(\Omega)$ for all $j$ in the range $-J\le j \le J$. Then, there exists a positive constant $c>0$, such that for all $s \in \mathbb C$ with  $\Re s > c\rho/\varepsilon,$ we have the following bound.
The coefficients $\widehat z^K_k$, with $k$ in the appropriate range $0\le k \le K$, decay with respect to $k$ in the sense that
\begin{align*}
   \|\widehat z^{JK}_{Jk} \|_1
\le  
 C_K\left(\dfrac{\rho}{\eps \Re s}+ \rho \dfrac{|s|}{\Re s} \right)^{k} \left\|\Delta^{k+1} \widehat f\right\|,
\end{align*}
where the constant $C$ depends only on $K$ and the coefficients $\mu_0$ and $\hat{\mu}_j$ for $-J\le j \le J$. For negative indices $k<0$, the analogous bound holds with the corresponding exponents $|k|$.
\end{proposition}
\begin{proof}
The proof is given by repeating the proof of Proposition~\ref{prop:decay-no-spatial-reg} (where the resolvent bound is used for every $J$\textsuperscript{th} component of $z^{JK}$ repeatedly).
\end{proof}
\begin{remark}
Stronger decay conditions, along the lines of Proposition~\ref{prop:decay-spatial-reg}, might hold for spatially variable modulations as well, but are beyond the scope of this paper. 
\end{remark}



\section{Error bounds for the modulated Fourier expansion}\label{sect:error bounds}
Error bounds for the modulated Fourier expansion are obtained by the following structure: First, we carry the Laplace domain bounds of the previous section over to the time-domain to obtain bounds for the coefficients $z^K$. Secondly, we introduce energy estimates for the time-modulated acoustic wave equation, for the studied setting. Finally, we insert the truncated modulated Fourier expansion into the time-modulated acoustic wave equation \eqref{eq:ac-intro} and apply the energy estimates. The following section prepares this treatment by introducing the necessary notation to carry over Laplace domain bounds to the time-domain. 

\subsection{Temporal convolution operators}\label{sect:temporal_convolution}
Let $K(s)\colon V \rightarrow W$ be an analytic family of bounded linear operators, defined for all $s\in\mathbb C$ with $\Re s>\sigma_0$ for some $\sigma_0$. Here, $V$ and $W$ are generally abstract Hilbert spaces, and will be set to $L^2(\Omega)$, $H_0^1(\Omega)$ or $V_K$ in the following sections.
The bounds of the previous section are polynomial with respect to the Laplace domain parameter $s$, in the following sense: There exists a real $\kappa$, $\nu\ge 0$ and, for every $\sigma >\sigma_0$, there exists $M_\sigma <\infty$, such that
\begin{align}\label{eq:pol_bound}
\norm{K(s)}_{V'\leftarrow V}&\leq M_\sigma \dfrac{\abs{s}^\kappa}{(\Re s)^\nu}, \quad\ \text{ Re } s \ge \sigma>\sigma_0>0.
\end{align}
For regular time-dependent functions $g:[0,T]\to V$ that initially vanish together with their first $m>\kappa$ derivatives, we use the Heaviside notation of operational calculus, which writes
\begin{equation} \label{Heaviside}
K(\partial_t) g = (\mathcal{L}^{-1}K) * g .
\end{equation}
This temporal convolution of the inverse Laplace transform of $K$ with~$g$, where $g$ is extended by zero on the negative real axis, is then a time-domain formulation of applying the analytic operator family $K(s)$ to the Laplace transform of $g$. 

Concatenating two analytic families $K(s)$ and $L(s)$ acting on compatible spaces in the Laplace transform is equivalent to concatenating the two corresponding temporal convolution operators, which is the composition rule
\begin{equation}\label{comp-rule}
K(\partial_t)L(\partial_t)g = (K L)(\partial_t)g.
\end{equation}
We denote the Sobolev space of real   order $r$ of $V$-valued functions on $\mathbb R$ by $H^r(\mathbb R,V)$ and, on finite intervals $(0, T )$,
we write
$$
H_0^r(0,T; V) = \{ g|_{(0,T)} \,:\, g \in H^r(\mathbb R, V)\ \text{ with }\ g = 0 \ \text{ on }\ (-\infty,0)\} .
$$
The natural norm on $H_0^r(0,T;V)$ is, for integer-valued $r\ge 0$, denoted by
\begin{align*}
\| g\|^2_{H_0^k(0,T;V)}= \sum_{\ell=0}^k \int_0^T\| \partial_t^\ell g (t)\|_V^2 \, \mathrm d t,
\end{align*}
for $g\in H_0^r(0,T; V)$. On this space, this norm is equivalent to the expression $\| \partial_t^r g \|_{L^2(0,T;V)}$, with constants that depend on the final time $T$.
By the Plancherel formula, the temporal convolution operators $K(\partial_t)$ are bounded on the appropriate temporal Sobolev spaces, which is shown in \cite[Lemma 2.1]{L94}:
If $K(s)$ is bounded by \eqref{eq:pol_bound} in the half-plane $\text{Re }s > \sigma>0$, then $K(\partial_t)$ fulfills
\begin{equation}\label{sobolev-bound}
\| K(\partial_t) \|_{ H^{r}_0(0,T;V') \leftarrow H^{r+\kappa}_0(0,T;V)} \le e^{\sigma T} M_{\sigma} T^{\nu},
\end{equation}
for arbitrary real $r$, which we further assume to be positive. 
 To derive estimates that are not expressed as a temporal semi-norm, we rely on the continuous embedding $H^{k+\alpha}_0(0,T;V)\subset C^k([0,T];V)$, which holds for arbitrary integer-valued $k\ge 0$ and real ${\alpha>\tfrac12}$.

\subsubsection*{Well-posedness for the modulated Fourier expansion}
In terms of the Heaviside notation of operational calculus, we can formulate the time-dependent formulation of \eqref{eq:trunc-sys-th} as: Find, for all $t\in [0,T]$, the set of $z^K(t)\in V_K$, such that 
\begin{align}\label{eq:trunc-sys-td}
 (D_K^2 (\partial_t) + \mathcal{T}_\Delta )  z^K 
 &=  f^K.
\end{align} 
As before in the section dedicated to Laplace domain bounds, we use the notation $ f^K = \left(\delta_{k0} f \right)_{-K\le k \le K}$ for the right-hand side.
By the composition rule of temporal convolution operators, the solution coefficients can further be characterized by 
\begin{align}\label{eq:trunc-sys-inv-td}
 z^K 
 &=  (D_K^2 (\partial_t) + \mathcal{T}_\Delta )^{-1} f^K,
\end{align} 
where the spatio-temporal operator on the right-hand side is the temporal convolution operator associated to the resolvent of the Laplace domain equation \eqref{eq:trunc-sys-th}.
\subsection{Well-posedness and bounds on the coefficients $z^K$}

First, we formulate a well-posedness result for the time-dependent solution of the truncated modulated Fourier expansion \eqref{eq:MFE}.
\begin{proposition}\label{prop:well-posedness-td}
Let $f\in H^r_0 (0,T; L^2(\Omega))$ for some $r \ge 0$.
Then, the system \eqref{eq:weak-K} (or equivalently \eqref{eq:trunc-sys-inv-td}) has the unique solution $$z^K \in H_0^r(0,T; H_0^1(\Omega)),$$ which is bounded by
\begin{align*}
\| z^K \|_{ H^r_0 (0,T; H_0^1(\Omega))}
 \le  2T\, e^{T 4\rho/\eps  }   \, \|  f \|_{ H^r_0 (0,T; L^2(\Omega))}.
\end{align*}
\end{proposition}
\begin{proof}
The statement is the direct consequence of combining Proposition~\ref{prop:well-posedness} with \eqref{sobolev-bound}.
\end{proof}
\noindent As a consequence of the Sobolev embedding  $H^{1}_0(0,T;V)\subset C^0([0,T];V)$, we note that Proposition~\ref{prop:well-posedness-td} implies the bound
\begin{align*}
\sup_{t\in[0,T]}\| z^K(t) \|_{1}
 \le    2 T  e^{T 4\rho/\eps  }\|  f \|_{ H^1_0 (0,T; L^2(\Omega))},
\end{align*}
as long as the temporal regularity of $f$ is sufficiently large for the right-hand side to be bounded.

\begin{proposition}\label{prop:decay-zK}

For any $k$ in the appropriate range $-K\le k \le K$, the coefficients $\widehat z^K_k$ decay with respect to $|k|$ in the sense that
\begin{align*}
\sup_{t\in[0,T]} \| z^K_k (t)\|_1
 &\le  C \rho^{|k|}\eps^{|k|}  T^{2|k|+1} 
  e^{T 4\rho/\eps  }
 \left\| \Delta^{k+1}  f\right\|_{H^{k+1}_0 (0,T; L^2(\Omega))},
\end{align*}
as long as the right-hand side is bounded. 
When the right-hand side does not fulfill any additional spatial regularity, i.e. $f\in H^r_0 (0,T; L^2(\Omega))$, then we still have the estimate
\begin{align*}
\sup_{t\in[0,T]} \| z^K_k (t)\|_1
 &\le  C \dfrac{\rho^{|k|}}{\eps^{|k|}}   T^{|k|+1}  e^{T 4\rho/\eps  }
 \left\|  f\right\|_{H^{k+1}_ 0 (0,T; L^2(\Omega))}.
\end{align*}
The constant $C$ depends only on $K$ and the statement holds again, as long as the right-hand side is bounded.
\end{proposition}
\begin{proof}
The statement is the direct consequence of the Laplace domain bounds of Propositions~\ref{prop:decay-no-spatial-reg}--\ref{prop:decay-spatial-reg} with the Plancherel formula, as formulated in \cite[Lemma 2.1]{L94}, in combination with the Sobolev embedding $C^1([0,T]) \subset H^1([0,T])$.
\end{proof}
\begin{remark}
Analogously, the following bound for the Laplace operator applied to the coefficients holds, namely for $k\ge 0$ we have
\begin{align}\label{eq:decay-rate-td}
   \|\Delta  z^K_k \|_{H_0^1(0,T;L^2(\Omega))}
\le  
  C \rho^k\eps^k T^{2k+1}  e^{T 4\rho/\eps  }\left\|\Delta^{K+1} 
  f\right\|_{H_0^{K+1}(0,T;L^2(\Omega))},
  \end{align}
  where the constant $C$ depends only on $K$.
\end{remark}
\subsection{Energy estimates for the time-modulated wave equation}
This section presents energy estimates for the time-modulated acoustic wave equation. 
 The energy of the acoustic wave, at a given time $t$, is given by 
	\begin{align}\label{energy}
		\mathcal E(t) = 	\frac{1}{2} 
		\int_D | \partial_t	u (x,t) |^2
		+\mu(x,t/\eps )\left|\nabla u (x,t)\right|^2 \, \mathrm d x.
	\end{align}
    Whenever convenient, the omnipresent spatial variable $x$ is dropped in the notation. 
    The following Lemma gives an identity for the variation of energy for sufficiently regular solutions, for which we use the formulation of \cite[Lemma~2.1]{NHA24}.
\begin{lemma}\label{lem:energy-identity-with-excit} Let $u\in C^{2}(0,T;H^1(D)) $ solve the evolution problem \eqref{eq:ac-intro} and further let $\mu$ be scaled as in \eqref{eq:ac-intro}. 
		Then we have, for all $t>0$, the energy identity
		\begin{align*}
			\mathcal E(t)- \mathcal E(0)
			&=
			\int_0^t \int_D  \partial_t u (x,t') f (x,t')+ \dfrac{\mu'(x,t'/\eps)}{\eps } \left|\nabla u (x,
            t')\right|^2\, \mathrm d x
			  \,\mathrm d t'
			.
		\end{align*}
	\end{lemma}
	Combining this identity with  Grönwall's Lemma gives the following bound, where we take the formulation of \cite[Proposition~2.1]{NHA24}.
    \begin{proposition}\label{prop:energy-estimate}
In the setting of Lemma~\ref{lem:energy-identity-with-excit}, we have the energy estimate
        \begin{align*}
        \mathcal{E}(t) 
        \le 
        e^{C_\mu' (t+\eps)} \left(\mathcal{E}(0)
        + \dfrac{1}{4 C_{\mu}'} \int_0^t  \| f(t')\|^2_{L^2(\Omega)}  \, \mathrm d t' \right),
        \end{align*}
		with the constant 
		\begin{align*}
			C_{\mu}' = \dfrac{2}{2 \pi \eps^2}\int_0^{2\pi \eps } \left\| \dfrac{\mu'(t'/\eps)_+}{\mu(t/\eps) } \right\|_{L^{\infty}(\Omega)}\, \mathrm d t'
            \, \, \propto \, \, \dfrac{\rho}{\eps}.
		\end{align*}	
      
	\end{proposition}
    \begin{remark}
      In the case of $\mu(t/\eps) = 1+2\rho \cos(t/\eps)$, the constant is given by
    	\begin{align*}
			C_{\mu}' = \dfrac{4 \rho }{2 \pi \eps} \left( \int_0^{\pi  } \sin(t)(1- 2\rho \cos(t))  \mathrm d t' + \mathcal O ( \rho^2 ) \right)  =\dfrac{\rho}{\eps}\left(\dfrac{4}{\pi} + \mathcal O(\rho^2) \right).
		\end{align*}	
    \end{remark}

\subsection{Proof of error bounds}
Inserting the modulated Fourier expansion into the time-modulated acoustic wave equation \eqref{eq:ac-intro} gives a characterization of the remainder term of \eqref{eq:modul-Fourier}. Combining the decay conditions of the coefficients $z^K$, as they are formulated in Proposition~\ref{prop:decay-zK}, with the energy estimate of Proposition~\ref{prop:energy-estimate} gives the following error bound. 
\begin{proposition}\label{prop:error-bound-mfe}
Let $\Delta^{K+1}f\in H^{K+1}_0 (0,T; L^2(\Omega))$. 
The remainder term $R_K$ of \eqref{eq:modul-Fourier} is then bounded by
\begin{align*}
\sup_{t\in[0,T]} \| R_K (t)\|_1
 &\le  C \rho^{K}\eps^{K+1}  e^{T 4\rho/\eps  } T^{2K+1}
 \left\| \Delta^{K+1} f\right\|_{H^{K+1}_ 0 (0,T; L^2(\Omega))},
\end{align*}
where the constant $C$ only depends on $K$.
\end{proposition}

\begin{proof}
Inserting the modulated Fourier expansion into the time-modulated acoustic wave equation yields, by construction of the coefficients $z_k^K$
\begin{align*}
(\partial_t^2-(1+2\rho \cos(t/\eps))\Delta )\sum_{k=-K}^K z_k^K(x,t) e^{i kt/\varepsilon}
&= f+
\rho \Delta \left( z_{-K}^Ke^{i (-K-1)t/\varepsilon}+z_{K}^{K} e^{i (K+1)t/\varepsilon} \right) .
\end{align*}
Subtracting the time-modulated acoustic wave equation therefore yields the following error equation for the  remainder term $R_K$
\begin{align*}
(\partial_t^2-(1+2\rho \cos(t/\eps))\Delta ) \dot{R}_K
&= 
\rho \Delta \left( z_{-K}^Ke^{i (-K-1)t/\varepsilon}+z_{K}^{K} e^{i (K+1)t/\varepsilon} \right)  .
\end{align*}
The statement is now the direct consequence of the energy estimate of Proposition~\ref{prop:energy-estimate} combined with \eqref{eq:decay-rate-td}.

\end{proof}

\section{Structural implications of the MFE}\label{sect:struct-pres}
The modulated Fourier expansion has implications on the exact solution of the time-modulated acoustic wave equation, which is explored in the following two subsections.

\subsection{An unexpected quasi-conservation property for short times}
Using the change of energy as quantified in Lemma~\ref{lem:energy-identity-with-excit} as the foundation of a Grönwall argument ignores cancellations by the oscillatory nature of the modulation. By using these cancellations, we obtain a stronger quasi-preservation of the energy over short times. 
\begin{proposition}\label{prop:change-energy}
Consider the time-modulated acoustic wave equation with the modulation $\eqref{eq:modul-ac}$. Let the right-hand side fulfill $\Delta f\in H_0^2(0,\widetilde T; L^2(\Omega))$ be supported in $[0,\widetilde T]$ (i.e. vanishing for $t>\widetilde T$) and further let $T> \widetilde  T$ be some later time. 
Then, there exists a constant $C$, such that 
\begin{align*}
\left|\mathcal E( T)
-\mathcal E (\widetilde T) \right|
& \le
C\rho  
  e^{C T \rho/\eps  }
 .
\end{align*}
The constant $C$ depends only on the geometry $\Omega$, on the final time $T$ and the right-hand side $f$, but is independent of $\rho$ and $\eps$.
\end{proposition}
\begin{proof}
Inserting the modulated Fourier expansion into the energy identity gives, after applying the error bound of Proposition~\ref{prop:error-bound-mfe}, the estimate
		\begin{align*}
			&\mathcal E(\widetilde T)- \mathcal E (T)
			= 
			\int_T^{\widetilde T} \int_\Omega \left( \dfrac{d}{dt} \mu(t/\eps) \right) \left|\nabla u (t)\right|^2  \mathrm d x \,
 \mathrm d t	
  =  \sum_{j,k=-1}^1
 			\int_T^{\widetilde T}  \int_\Omega  \dfrac{\mu'(t/\eps)}{\eps} e^{i(j-k)t/\eps} \nabla \overline{z}_k^{K}\cdot \nabla z_j^K \mathrm d x  \mathrm d t
            +\mathcal O \left(\rho^3\eps\right)
            .
    \end{align*}
   The derivative of the modulation \eqref{eq:ref-mu-cos} has the form 
   $$\dfrac{\mu'(t/\eps)}{\eps} = -\dfrac{\rho}{\eps} \sin (t/\eps) = \dfrac{\rho}{\eps} \dfrac{ e^{-it/\eps}-e^{it/\eps} }{2i}.$$
Inserting this expression into the right-hand side gives
    \begin{align*}
\mathcal E(\widetilde T)- \mathcal E (T) & = \dfrac{\rho}{ \eps} \sum_{j,k=-1}^1
 			\int_T^{\widetilde T}  \int_\Omega   \dfrac{e^{i(j-k- 1)t/\eps} -e^{i(j-k+ 1)t/\eps}}{i} \nabla \overline{z}_k^K\cdot \nabla z_j^K \, \mathrm d x  \mathrm dt 
            +\mathcal O \left(\rho^3\eps \right).
\end{align*}
Each summand containing a highly oscillatory factor in the integrand is of the order $\mathcal{O}(\rho)$, due to the high temporal regularity of the coefficient functions $z_j^K$ shown in Proposition~\ref{prop:well-posedness-td}. We therefore consider now only those terms in the leading order that contain highly oscillatory terms and include all other terms in the remainder term, which gives
\begin{align*}
\left|\mathcal E(\widetilde T)- \mathcal E (T) \right| &= \dfrac{\rho}{\eps} \left|
 \sum_{k=-K}^{K-1}
 			\int_T^{\widetilde T}  \int_\Omega    \nabla \overline{z}_k^K\cdot\nabla z^K_{k+1}
            -\nabla \overline{z}_k^K\cdot \nabla z^K_{k-1}\, \mathrm d x \, \mathrm dt  \right|
            +\mathcal O \left(\rho \right)
\\& \le  i\dfrac{\rho}{\eps} 
 \sum_{k=-1}^{1}
 			\int_T^{\widetilde T}   \left\|z_k^K\right\|_1 \left\|z^K_{k+1}\right\|_1+\left\|z_k^K\right\|_1  \left\|z^K_{k-1}   \right\|_1
         \mathrm dt  
            +\mathcal O \left(\rho \right).
\end{align*}
The statement is then given by inserting the decay bounds \eqref{eq:decay-rate-td} of the coefficients $z_k^K$.
\end{proof}

\subsection{Invariant of the MFE}
We consider the coupled system for $f=0$, which reads for all $-K\le k\le K$:
\begin{align}\label{eq:coupled-system-f-0}
 (\partial_t+ ik/\varepsilon)^2 z^K_k  - \Delta z^K_k-\rho\Delta (z^K_{k-1}+z^K_{k+1}) =  0.
 \end{align}
For the sake of this section, we assume that the coefficients either fulfill some nontrivial initial conditions (which are symmetric in the sense of Remark~\ref{rem:symmetry}), or solve the dynamical system with some $f$ that vanishes on the time interval of interest. As long as the excitation vanishes, the system has a conserved property, which is explored in the following propostion.

 \begin{proposition}\label{prop:conservation}
The solution of the coupled system \eqref{eq:coupled-system-f-0}, with vanishing right-hand side $f$, conserves the following quantity $E\in \mathbb R$, namely the expression
\begin{align*}
E = \sum_{k=-K}^K \|\dot{z}^K_k \|^2 - \dfrac{k^2}{\varepsilon^2}\|z^K_k \|^2
+\|z^K_k \|_1^2
 +\rho \, \Re\left( \nabla z^K_k, \nabla z^K_{k-1} \right)
\end{align*}
is constant.
 \end{proposition}
\begin{proof}

Testing the $k$-th component of \eqref{eq:MFE} with $\dot{z}^K_k = \partial_t z^K_k$ and the summation over all $k$ gives 
 \begin{align}\label{eq:MFE-tested}
 \sum_{k=-K}^K \left(\dot{z}^K_k,(\partial_t+ ik/\varepsilon)^2z^K_k \right)  +\left(\nabla \dot{z}^K_k, \nabla z^K_k\right)+\rho\left( \nabla \dot{z}^K_k, \nabla (z^K_{k-1}+z^K_{k+1}) \right) =  0.
 \end{align}
For the real part of each of the summands of the first component we obtain
\begin{align*}
\Re \left(\dot{z}^K_k,(\partial_t+ ik/\varepsilon)^2z^K_k \right)
=
\Re \left(\dot{z}^K_k,\partial_t^2z^K_k \right)
-
\Re \dfrac{k^2}{\varepsilon^2} \left(\dot{z}^K_k,z^K_k \right)
= \dfrac{1}{2} \dfrac{\mathrm d}{\mathrm d t} \left( \|\dot{z}^K_k \|^2 - \dfrac{k^2}{\varepsilon^2} \|z^K_k \|^2
\right).
\end{align*}
For the second summands of \eqref{eq:MFE-tested}, we obtain
\begin{align*}
    \left(\nabla \dot{z}^K_k, \nabla z^K_k\right)
    =
\dfrac{1}{2}\dfrac{\mathrm d}{\mathrm d t}\left\|\nabla z^K_k \right\|^2.
\end{align*}
Finally, for the third sum, we first use an index shift to obtain
\begin{align*}
 & \Re \sum_{k=-K}^K \left( \nabla \dot{z}^K_k, \nabla (z^K_{k-1}+z^K_{k+1}) \right)
=
 \Re\sum_{k=-K}^K \left( \nabla \dot{z}^K_k, \nabla z^K_{k-1} \right)
 +
 \left(  \nabla z^K_{k+1},\nabla \dot{z}^K_k\right)
 \\
 &=\
 \Re\sum_{k=-K+1}^K \left( \nabla \dot{z}^K_k, \nabla z^K_{k-1} \right)
 +
 \left(  \nabla z^K_{k},\nabla \dot{z}^K_{k-1}\right)
=
 \dfrac{\mathrm d}{\mathrm d t} 
 \Re\sum_{k=-K+1}^K \left( \nabla z^K_k, \nabla z^K_{k-1} \right).
\end{align*}
Note that we use the notation $z^K_{-K-1}=z^K_{K+1} = 0$.
\end{proof}
\begin{remark}
When $\eps$ is sufficiently large and $\rho$ sufficiently small, the second and fourth terms of the conservation law in Proposition~\ref{prop:conservation} can be absorbed. The conservation law of Proposition~\ref{prop:conservation} then implies that the solution of the coupled dynamical system \eqref{eq:MFE} stays bounded for arbitrary long times.
\end{remark}


\section{Time-discretization of coupled dynamic system}\label{sect:cq}
Since the functions $z^K$ and their derivatives are temporally regular (for sufficiently smooth right-hand sides), it is natural to approximate them instead of the oscillatory solution $u$ of the time-modulated acoustic wave equation. 
Let $t_n = n \tau$ and $z^{K,n}_k \approx z^K_k(t_n)$.
The trapezoidal rule applied to the coupled system \eqref{eq:MFE} then
reads
\begin{align}\label{eq:tr-mfe}
\begin{aligned}
   \frac{z^{K,n+1}_k - 2 z^{K,n}_k + z^{K,n-1}_k}{\Delta t^2}
+ \frac{2 i k}{\varepsilon}\frac{z^{K,n+1}_k - z^{K,n}_k}{\Delta t}
- \frac{k^2}{\varepsilon^2}\frac{z^{K,n+1}_k + z^{K,n}_k}{2}
  &\\
- \Delta \frac{z^{K,n+1}_k + z^{K,n}_k}{2}
- \dfrac{\rho}{2} \Delta \left(
z^{K,n+1}_{k-1} + z^{K,n}_{k-1}
+ z^{K,n+1}_{k+1} + z^{K,n}_{k+1}
\right)
&=
\delta_{k0}\,\frac{f(t_{n}) + f(t_{n+1})}{2}. 
\end{aligned}
\end{align}
Classical time stepping schemes of this type are equivalent to the convolution quadrature method applied to system \eqref{eq:trunc-sys-td}, (see e.g. \cite{L94}). Combining standard theory with the Laplace domain bounds of Section~\ref{sect:Lap} is thus a convenient and natural path to error bounds of the approximation. The following section briefly introduces these techniques to show error bounds for the above and similar time stepping schemes.

\subsection{Time stepping as convolution quadrature}
Convolution quadrature methods are extensions of classical time stepping scheme, which discretize the temporal convolution \eqref{Heaviside} by a discrete convolution with stepsize $\tau>0$. For a detailed introduction to such methods and their original derivation for wave-type problems we refer the reader to \cite{L94}. 
The approximation reads
\begin{align*}
	\left( K(\partial_t^\tau)g\right) (t):= \sum_{j\ge 0} \omega_j \, g(t-j\tau),
\end{align*}
which is defined for $0\le t \le T$, but usually considered only on the grid $t_n=n\tau$, so that only grid values of $g$ appear on the right-hand side.
The convolution quadrature weights are defined as the coefficients of the generating power series
\begin{align}\label{eq:cq-weights}
	\sum_{j=0}^{\infty}\omega_j \, \zeta^j:=K\left(\dfrac{\delta(\zeta)}{\tau}\right), 
\end{align}
where we choose $\delta(\zeta)$ as the generating polynomial of the trapezoidal rule 
$$\delta(\zeta) = \dfrac{2(1-\zeta)}{1+\zeta}.$$
The method is A-stable, namely we have $\Re \delta(\zeta)\ge 0$ for $|\zeta|\le 1$, which ensures that $K$ is only evaluated in its domain of analyticity when computing the weights \eqref{eq:cq-weights}.
A key property of this type is a discrete operational calculus: just as $K(\partial_t)L(\partial_t)g=(KL)(\partial_t)g$, 
we also have the composition rule
\begin{align}\label{eq:composition-rule-tau}
K(\partial_t^\tau)L(\partial_t^\tau)g=(KL)(\partial_t^\tau)g.
\end{align}

\subsection{Discretized coupled system}
Applying the time-discretization to the coupled system \eqref{eq:trunc-sys-td} then gives the temporally discrete scheme
\begin{align}\label{eq:trunc-sys-td-tau}
 (D_K^2 (\partial^\tau_t) + \mathcal{T}_\Delta )  z_\tau^K 
 &=  f^K,
\end{align}
By the composition rule \eqref{eq:composition-rule-tau}, this scheme is equivalent to invert the operator in the left-hand side, since the existence of the inverse of the Laplace domain counterpart is guaranteed by Proposition~\ref{prop:well-posedness-spatially-variable}. An explicit form of the approximation is therefore given by
\begin{align}\label{eq:trunc-sys-td-tau-inv}
  z_\tau^K 
 &=   (D_K^2 (\partial^\tau_t) + \mathcal{T}_\Delta )^{-1}f^K,
\end{align}
where the right-hand side is the convolution quadrature approximation to the resolvent of Equation~\ref{eq:trunc-sys-inv-td}.
\begin{remark}[Equivalence of formulations]
Comparing the generating functions of the scheme \eqref{eq:trunc-sys-td-tau} and the scheme \eqref{eq:tr-mfe} shows their mathematical equivalence. Computationally, the methods are different. The classical scheme \eqref{eq:tr-mfe} is (in combination with a standard space discretization) straightforward to implement and leads to a time-marching code that sequentially solves the approximations to the time steps $z^K(t_n))$, where the same linear system has to be inverted for many different right-hand sides. Alternatively, the scheme can be implemented via \eqref{eq:trunc-sys-td-tau-inv}, as a convolution quadrature method, as is discussed e.g. in \cite{B10}. Here, the Laplace domain system \eqref{eq:trunc-sys-th} has to be approximated at $\mathcal{O}(N)$ different values for $s\in\mathbb C_+$, which can be done in parallel but has some computational overhead.
\end{remark}
Error bounds now directly follow from the Laplace domain bounds and general convolution quadrature approximation results (i.e. \cite[Theorem A.2]{B10}).

\begin{theorem}\label{th:error-bounds-tau}
Consider the temporal coupled system \eqref{eq:trunc-sys-td} (or its multimodal counterpart \eqref{eq:general-mfe-sys}) and its trapezoidal rule discretization \eqref{eq:trunc-sys-td-tau} at a time $t_n=n\tau$, which fulfills $t_n \rho <C\eps$. Then, when the right-hand side $f$ fulfills sufficient temporal regularity, we have the estimate 
\begin{align}
\left\| z^K(t_n)- (z^K_\tau )_n  \right\|_1 \le C \tau^2 \| f\|_{H^4_0(0,T;L^2(\Omega))}.
\end{align}
The constant $C$ depends on the final time $T$ and the modulation $\mu$.
\end{theorem}
\begin{proof}
The proof is the direct consequence of the Laplace domain bound of Proposition~\ref{prop:well-posedness-spatially-variable} and the general approximation result of the convolution quadrature method based on the trapezoidal rule \cite[Theorem A.2]{B10}, which is applied to the approximation \eqref{eq:trunc-sys-td-tau-inv}.
\end{proof}
\begin{remark}[Higher-order methods and error bounds]
Due to the high temporal regularity of the coefficients $z^K$, which is ensured by Proposition~\ref{prop:well-posedness-td}, higher-order discretizations outperform the trapezoidal rule, as long as the right-hand side $f$ has sufficient temporal regularity and $\tau$ is sufficiently small. As the Dahlquist order barrier prohibits A-stable discretizations beyond order $2$, multistage methods are the natural class of interest and often outperform multistep-based schemes, in particular in the context of convolution quadrature methods (see \cite{B10}). Due to the Laplace domain analysis, error bounds for a general class of methods are the direct consequence of repeating the argument of Theorem~\ref{th:error-bounds-tau} with the appropriate general approximation result, where \cite[Theorem~3]{BLM11} can be used for stiffly accurate Runge--Kutta methods (such as the Radau IIA methods) or \cite{BF24} for Gauss based methods.
\end{remark}


\section{Numerical experiments}\label{sect:numerics}
The codes that were used to conduct the experiments and generate the plots are distributed via github\footnote{ \href{https://github.com/joergnick/ModulatedFourierExpansion}{https://github.com/joergnick/ModulatedFourierExpansion}, last accessed on 05.02.2026.}.
Our numerical investigations are based on the one-dimensional acoustic wave equation on the domain $\Omega=(0,1)$, with homogeneous Dirichlet boundary conditions. For the modulation, we restrict our attention to the spatially homogeneous modulation \eqref{eq:ref-mu-cos}. The system is excited by a balanced source on the right-hand side of the form
\begin{align}\label{eq:source}
f(x,t) = g(x,t-t_0)-g(x,t-t_0-0.1),
\end{align}
with the spatio-temporal gaussian $$g(x,t) = e^{-100(x-0.5)^2-10t^2} .$$ The opposite signs of the excitation has been used to excite a small wave packet around the time $t_0 = 1$ and the simulation observes the behavior of the wave equation until the final time $T=4$. We use the standard finite element method to discretize the spatial integrator $-\Delta$, with $1000$ degrees of freedom. As the error measure, we compute the $L^2-$norm in time and space, which reads
\begin{align*}
\text{ERR} = \left(\sum_{n=0}^N \|e_n\|^2\right)^{\frac{1}{2}}.
\end{align*}

\emph{Error behavior of time-integrator.} Figure~\ref{fig:time-convergence} shows the time convergence of two different integrators: firstly, the error of the standard trapezoidal rule applied to the time-modulated acoustic wave equation \eqref{eq:ac-intro}. As expected, the method is convergent only if the step size $\tau$ is smaller than the scale of the temporal oscillation $\varepsilon$. Secondly, we consider the error of the approximation that is induced by approximating the coefficients $z^K$ with the trapezoidal rule, i.e. the approximation of Equation~\ref{eq:trunc-sys-td-tau}. The coupled system is truncated at $K=3$, which ensures a truncation error that is sufficiently small to observe the convergence of the trapezoidal rule applied to the coupled dynamical system. To compute the error, we use a reference solution that has been computed with $N=2^{14}$ timesteps and the same space discretization.

\begin{figure}[h!]
	\centering 
	\includegraphics[scale=.75]{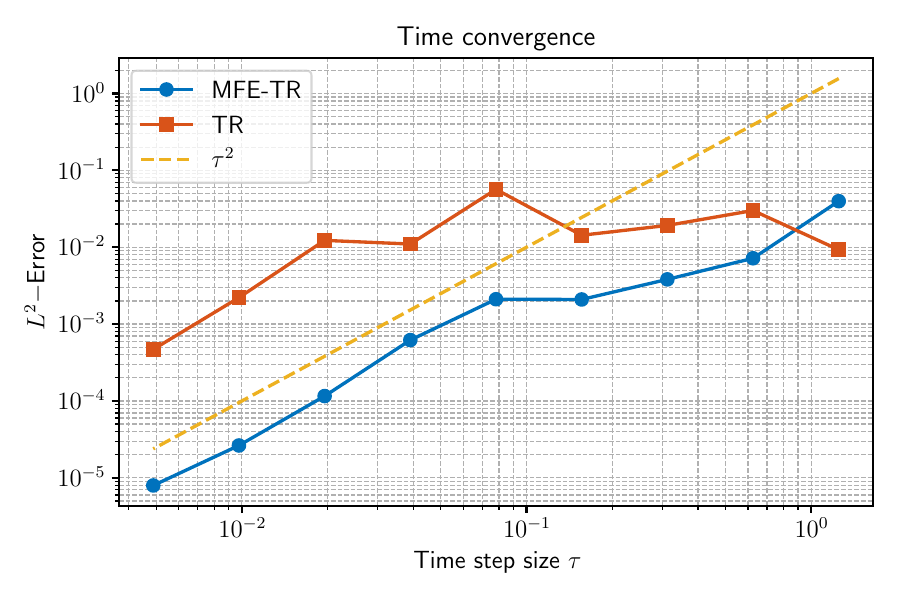}
	\caption{Time convergence error for $\rho = 0.4$ and $\eps = 0.04$, measured in the space time $L^2-$ norm in time and space. The solution is computed until the final time $T=5$ and the MFE is used with $K=3$. A finite difference solver is used for space discretization with $1000$ degrees of freedom.
    }
	\label{fig:time-convergence}
\end{figure}
\emph{Visualization of the solution.} Consider the solution $u$ in the same setting as before, namely with the regular right-hand side $f$ of \eqref{eq:source}, $\rho=0.4$ and $\eps=0.04$. Here, we use a space discretization with $500$ degrees of freedom, which is now combined by the trapezoidal time discretization for the dynamical coupled system that uses $N=2^8=256$ timesteps. Figure~\ref{fig:visualization} visualizes the approximation to the solution $u$ of the time-modulated acoustic wave equation. The solution displays oscillatory behavior, which becomes more apparent for the final time. We note that the visualization depicts the solution for larger times than are covered by the theory, as the constants are exponentially dependent on the factor $T\rho/\eps= 40$ at the final time. Nevertheless, the solution remains bounded to that time, for this specific example (Although some growth is observed in the example). The smoothly varying coefficient functions $z_k$ are visualized in Figure~\ref{fig:z_K-time}, which show their absolute value over time. 

\emph{Decay of the coefficient functions $z_k$.} The defect introduced by truncating the coupled dynamical system at a finite index $K$ is determined by the decay of the coefficient functions $z_k$.

Figure~\ref{fig:z_K} shows the decay for the regular right-hand side of \eqref{eq:source}, for several values of $\eps$ and $\rho$. As reference lines, we visualize the rates predicted by Theorem~\ref{thm:mfe}, namely $\rho^k\eps^k$. The negative indices $-k$ fulfill, due to the symmetry $z_{-k}=\overline{z}_k$, the same decay property. For all values of $\eps$ and $\rho$, we observe a remarkable agreement of the predicted theoretical results and the practical experiments. It should be noted that the theory does not imply a convergence result with respect to $K\rightarrow \infty$ (since the constants, e.g. in Theorem~ \ref{thm:mfe}, are not controlled with respect to $K$). 

High spatial regularity of the right-hand side $f$ is a critical component of the theory to show the high-order decay estimates of Theorem~\ref{thm:mfe}. To test the effect of a source with a low spatial regularity, we consider an excitation that is a step function in space at any time-point, which takes the form
\begin{align}\label{eq:f-low-reg}
    f(x,t) = 
\begin{cases}
    e^{-10(t-1)^2} ,&\text{for} \quad \frac{1}{4} \le x\ \  \text{or} \ \  x \le \frac{3}{4},
    \\   0 , &\text{else} .
\end{cases}
\end{align}
The decay of the values $z_k$ for this excitation of low spatial regularity (with $1000$ spatial degrees of freedom, $N=2^{12}=4096$ time steps until the final time $T=2$) is shown in Figure~\ref{fig:z_K_no_reg}. Remarkably, a much stronger decay than predicted by the theory is observed here (since the theoretical decay without any assumptions on the spatial regularity can only be estimated by Theorem~\ref{thm:mfe-low-reg}, predicting the exponential decay rate $\rho/\eps$). Nevertheless, the lack of high spatial regularity does affect the decay properties of $z_k$, whose exponential rate of decay slows down for larger $k$ notably compared to Figure~\ref{fig:z_K}.

\begin{figure}[h!] 
	\hspace*{2.2cm}
	\includegraphics[scale=.75]{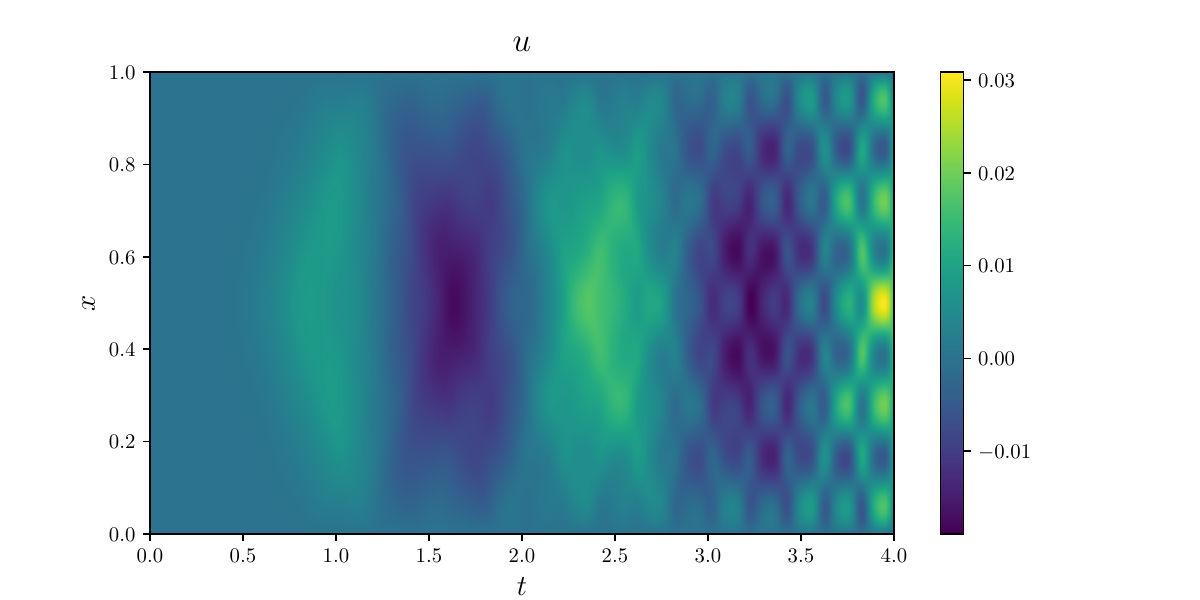}
	\caption{Visualization of $u$, with $500$ spatial degrees of freedom and $N=256$ time steps. The parameters were set to $\rho=0.4$ and $\eps=0.04$.}
	\label{fig:visualization}
\end{figure}
\begin{figure}[h!]
	\centering 
	\includegraphics[scale=.6]{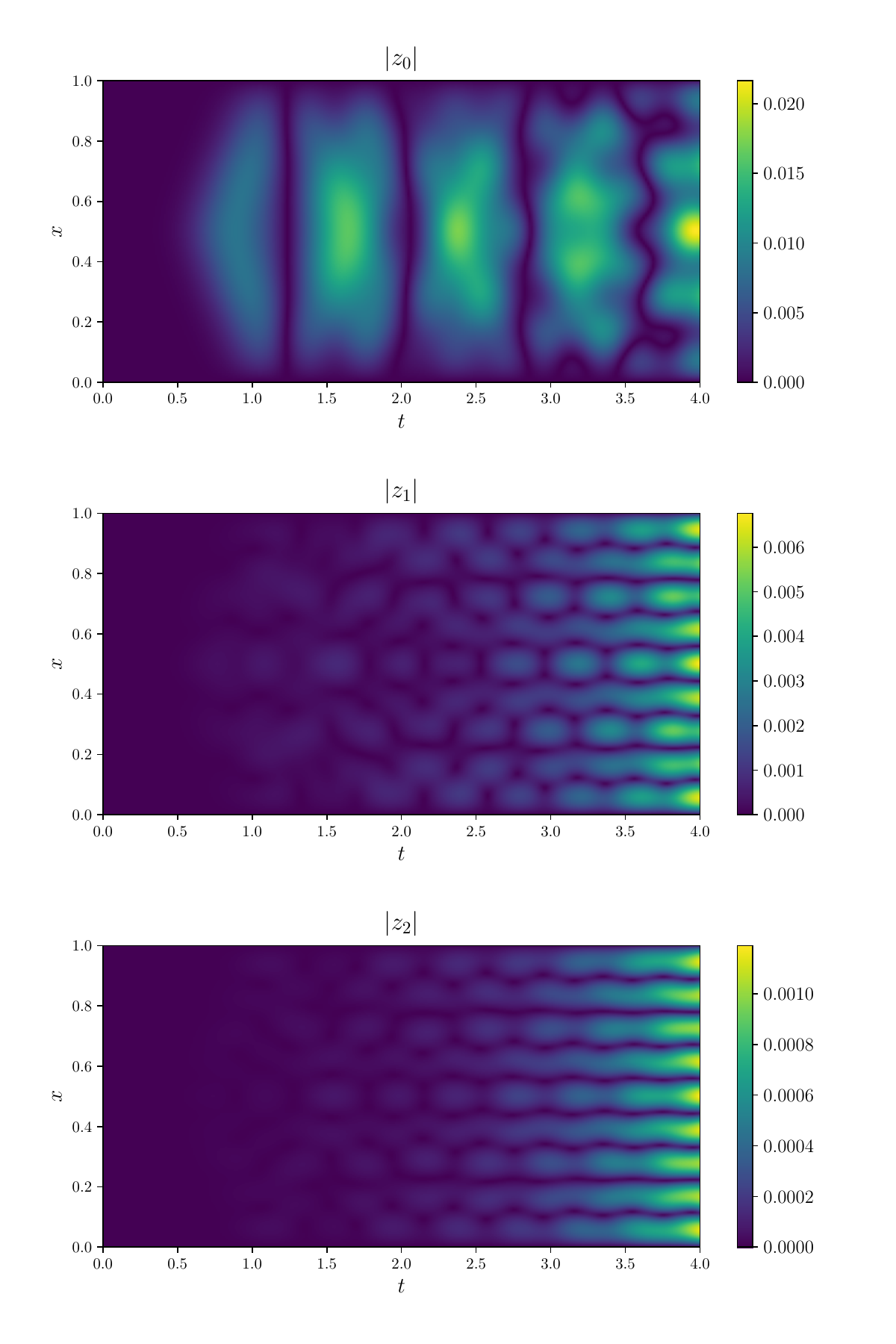}
	\caption{The absolute values of the coefficient functions $z_0,z_1$ and $z_2$ that were used to compute the approximation of Figure~\ref{fig:visualization}.}
	\label{fig:z_K-time}
\end{figure}

\begin{figure}[h!]
	\centering 
	\includegraphics[scale=.75]{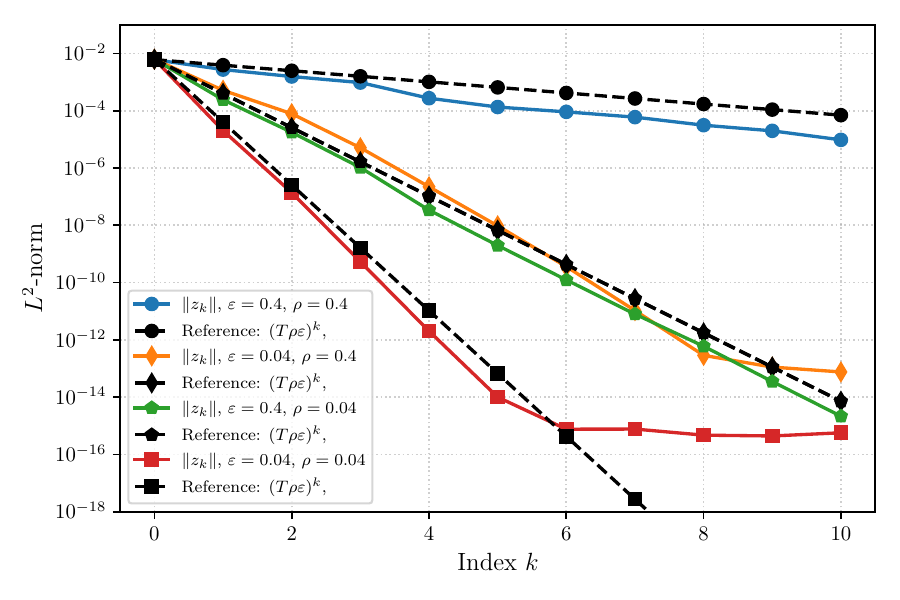}
	\caption{Spatio-temporal $L^2$-norm of the modulation functions $z_k$ for $k=0,\dots,K$ with $K=10$, for different values of $\eps$ and $\rho$, with the regular right-hand side $f$ of \eqref{eq:source}.}
	\label{fig:z_K}
\end{figure}

\begin{figure}[h!]
	\centering 
	\includegraphics[scale=.75]{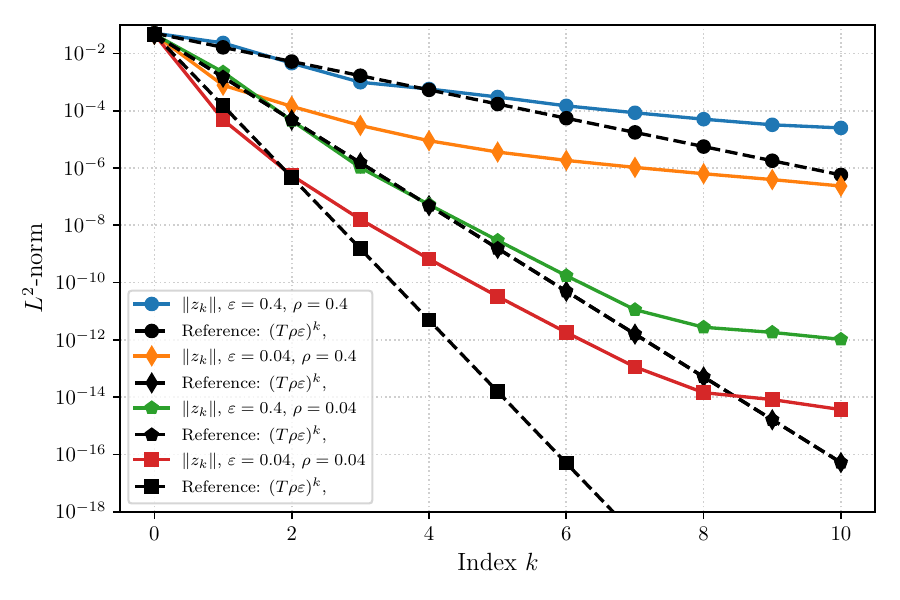}
	\caption{The same plot as Figure~\ref{fig:z_K}, but for the excitation $f$ described in Equation \ref{eq:f-low-reg}, with low spatial regularity. 
    }
	\label{fig:z_K_no_reg}
\end{figure}


\section*{Acknowledgement}
The author thanks Habib Ammari and Martin J. Gander for providing support and the freedom to the author to pursue this topic independently, as well as many helpful discussions. The author further thanks Christian Lubich for insightful discussions that introduced the author to modulated Fourier expansions. Finally, the author is grateful to Andreas Prohl and Bin Wang for helpful comments.





\bibliographystyle{abbrv}
\bibliography{Lit}

@article {CHL08,
    AUTHOR = {Cohen, David and Hairer, Ernst and Lubich, Christian},
     TITLE = {Long-time analysis of nonlinearly perturbed wave equations via
              modulated {F}ourier expansions},
   JOURNAL = {Arch. Ration. Mech. Anal.},
  FJOURNAL = {Archive for Rational Mechanics and Analysis},
    VOLUME = {187},
      YEAR = {2008},
    NUMBER = {2},
     PAGES = {341--368},
      ISSN = {0003-9527,1432-0673},
   MRCLASS = {35L70 (35B10 35B40)},
  MRNUMBER = {2366141},
MRREVIEWER = {Hideo\ Kubo},
       DOI = {10.1007/s00205-007-0095-z},
       URL = {https://doi.org/10.1007/s00205-007-0095-z},
}

@article {L94,
	AUTHOR = {Lubich, Ch.},
	TITLE = {On the multistep time discretization of linear
	initial-boundary value problems and their boundary integral
	equations},
	JOURNAL = {Numer. Math.},
	FJOURNAL = {Numerische Mathematik},
	VOLUME = {67},
	YEAR = {1994},
	NUMBER = {3},
	PAGES = {365--389},
	ISSN = {0029-599X,0945-3245},
	MRCLASS = {65M15 (65D30 65M60 65R20)},
	MRNUMBER = {1269502},
	MRREVIEWER = {Rudolf\ Gorenflo},
	DOI = {10.1007/s002110050033},
	URL = {https://doi.org/10.1007/s002110050033},
}

@article{TLAGC24,
author = {Touboul, Marie and Lombard, Bruno and Assier, Raphael C and Guenneau, Sebastien and Craster, Richard V},
title = {High-order homogenization of the time-modulated wave equation: non-reciprocity for a single varying parameter},
journal = {Proceedings of the Royal Society A: Mathematical, Physical and Engineering Sciences},
volume = {480},
number = {2289},
pages = {20230776},
year = {2024},
doi = {10.1098/rspa.2023.0776},

URL = {https://royalsocietypublishing.org/doi/abs/10.1098/rspa.2023.0776},
eprint = {https://royalsocietypublishing.org/doi/pdf/10.1098/rspa.2023.0776}
}

@article{YGA22,
	title={Floquet metamaterials},
	author={Yin, Shixiong and Galiffi, Emanuele and Al{\`u}, Andrea},
	journal={ELight},
	volume={2},
	number={1},
	pages={8},
	year={2022},
	publisher={Springer}
}

@article {ACH22,
	AUTHOR = {Ammari, Habib and Cao, Jinghao and Hiltunen, Erik Orvehed},
	TITLE = {Nonreciprocal wave propagation in space-time modulated media},
	JOURNAL = {Multiscale Model. Simul.},
	FJOURNAL = {Multiscale Modeling \& Simulation. A SIAM Interdisciplinary
	Journal},
	VOLUME = {20},
	YEAR = {2022},
	NUMBER = {4},
	PAGES = {1228--1250},
	ISSN = {1540-3459,1540-3467},
	MRCLASS = {35J05 (35C20 35P20 74J20)},
	MRNUMBER = {4500526},
	DOI = {10.1137/21M1449427},
	URL = {https://doi.org/10.1137/21M1449427},
}

@article{TK19,
  title={Floquet engineering of quantum materials},
  author={Oka, Takashi and Kitamura, Sota},
  journal={Annual Review of Condensed Matter Physics},
  volume={10},
  number={1},
  pages={387--408},
  year={2019},
  publisher={Annual Reviews}
}

@article {MV76,
    AUTHOR = {Miranker, W. L. and van Veldhuizen, M.},
     TITLE = {The method of envelopes},
   JOURNAL = {Math. Comp.},
  FJOURNAL = {Mathematics of Computation},
    VOLUME = {32},
      YEAR = {1978},
    NUMBER = {142},
     PAGES = {453--496},
      ISSN = {0025-5718,1088-6842},
   MRCLASS = {65L05 (34E15)},
  MRNUMBER = {494952},
MRREVIEWER = {W.\ S.\ Loud},
       DOI = {10.2307/2006158},
       URL = {https://doi.org/10.2307/2006158},
}

@article {HL00,
    AUTHOR = {Hairer, Ernst and Lubich, Christian},
     TITLE = {Long-time energy conservation of numerical methods for
              oscillatory differential equations},
   JOURNAL = {SIAM J. Numer. Anal.},
  FJOURNAL = {SIAM Journal on Numerical Analysis},
    VOLUME = {38},
      YEAR = {2000},
    NUMBER = {2},
     PAGES = {414--441},
      ISSN = {0036-1429,1095-7170},
   MRCLASS = {65L05 (34C15 37M99 65P10)},
  MRNUMBER = {1770056},
MRREVIEWER = {Riccardo\ Fazio},
       DOI = {10.1137/S0036142999353594},
       URL = {https://doi.org/10.1137/S0036142999353594},
}

@book {HLW02,
    AUTHOR = {Hairer, Ernst and Lubich, Christian and Wanner, Gerhard},
     TITLE = {Geometric numerical integration},
    SERIES = {Springer Series in Computational Mathematics},
    VOLUME = {31},
      NOTE = {Structure-preserving algorithms for ordinary differential
              equations},
 PUBLISHER = {Springer-Verlag, Berlin},
      YEAR = {2002},
     PAGES = {xiv+515},
      ISBN = {3-540-43003-2},
   MRCLASS = {65P10 (37J05 37M15 70-08 70H05)},
  MRNUMBER = {1904823},
MRREVIEWER = {Antonella\ Zanna},
       DOI = {10.1007/978-3-662-05018-7},
       URL = {https://doi.org/10.1007/978-3-662-05018-7},
}

@book {HW96,
    AUTHOR = {Hairer, E. and Wanner, G.},
     TITLE = {Solving ordinary differential equations. {II}},
    SERIES = {Springer Series in Computational Mathematics},
    VOLUME = {14},
   EDITION = {Second},
      NOTE = {Stiff and differential-algebraic problems},
 PUBLISHER = {Springer-Verlag, Berlin},
      YEAR = {1996},
     PAGES = {xvi+614},
      ISBN = {3-540-60452-9},
   MRCLASS = {65-02 (34A09 34A45 65-01 65Lxx)},
  MRNUMBER = {1439506},
       DOI = {10.1007/978-3-642-05221-7},
       URL = {https://doi.org/10.1007/978-3-642-05221-7},
}

@article {HW09,
    AUTHOR = {Hirosawa, Fumihiko and Wirth, Jens},
     TITLE = {Generalised energy conservation law for wave equations with
              variable propagation speed},
   JOURNAL = {J. Math. Anal. Appl.},
  FJOURNAL = {Journal of Mathematical Analysis and Applications},
    VOLUME = {358},
      YEAR = {2009},
    NUMBER = {1},
     PAGES = {56--74},
      ISSN = {0022-247X,1096-0813},
   MRCLASS = {35L15 (35L65)},
  MRNUMBER = {2527581},
MRREVIEWER = {Alexey\ V.\ Borovskikh},
       DOI = {10.1016/j.jmaa.2009.04.048},
       URL = {https://doi.org/10.1016/j.jmaa.2009.04.048},
}

@article {CD03,
    AUTHOR = {Colombini, Ferruccio and Del Santo, Daniele and Reissig,
              Michael},
     TITLE = {On the optimal regularity of coefficients in hyperbolic
              {C}auchy problems},
   JOURNAL = {Bull. Sci. Math.},
  FJOURNAL = {Bulletin des Sciences Math\'ematiques},
    VOLUME = {127},
      YEAR = {2003},
    NUMBER = {4},
     PAGES = {328--347},
      ISSN = {0007-4497,1952-4773},
   MRCLASS = {35L15 (35B30 35B65 35D10)},
  MRNUMBER = {1988632},
MRREVIEWER = {Florin\ Iacob},
       DOI = {10.1016/S0007-4497(03)00025-3},
       URL = {https://doi.org/10.1016/S0007-4497(03)00025-3},
}

@techreport{K58,
  title={The gyration of a charged particle},
  author={Kruskal, Martin},
  year={1958},
  institution={Princeton Univ., NJ Project Matterhorn}
}

@article {B10,
    AUTHOR = {Banjai, Lehel},
     TITLE = {Multistep and multistage convolution quadrature for the wave
              equation: algorithms and experiments},
   JOURNAL = {SIAM J. Sci. Comput.},
  FJOURNAL = {SIAM Journal on Scientific Computing},
    VOLUME = {32},
      YEAR = {2010},
    NUMBER = {5},
     PAGES = {2964--2994},
      ISSN = {1064-8275,1095-7197},
   MRCLASS = {65M38 (35L05 35L20 65L06 65M20)},
  MRNUMBER = {2729447},
MRREVIEWER = {Luigi\ Brugnano},
       DOI = {10.1137/090775981},
       URL = {https://doi.org/10.1137/090775981},
}

@incollection {RY00,
    AUTHOR = {Reissig, Michael and Yagdjian, Karen},
     TITLE = {Klein-{G}ordon type decay rates for wave equations with
              time-dependent coefficients},
 BOOKTITLE = {Evolution equations: existence, regularity and singularities
              ({W}arsaw, 1998)},
    SERIES = {Banach Center Publ.},
    VOLUME = {52},
     PAGES = {189--212},
 PUBLISHER = {Polish Acad. Sci. Inst. Math., Warsaw},
      YEAR = {2000},
   MRCLASS = {35L15 (35B45)},
  MRNUMBER = {1773102},
MRREVIEWER = {Tohru\ Ozawa},
}

@article {GH25,
    AUTHOR = {Goto, Kazunori and Hirosawa, Fumihiko},
     TITLE = {On the energy decay estimate for the dissipative wave equation
              with very fast oscillating coefficient and smooth initial
              data},
   JOURNAL = {J. Math. Anal. Appl.},
  FJOURNAL = {Journal of Mathematical Analysis and Applications},
    VOLUME = {548},
      YEAR = {2025},
    NUMBER = {1},
     PAGES = {Paper No. 129386, 36},
      ISSN = {0022-247X,1096-0813},
   MRCLASS = {35L15 (35B40)},
  MRNUMBER = {4867828},
MRREVIEWER = {Kenji\ Nishihara},
       DOI = {10.1016/j.jmaa.2025.129386},
       URL = {https://doi.org/10.1016/j.jmaa.2025.129386},
}

@article {RS05,
    AUTHOR = {Reissig, Michael and Smith, James},
     TITLE = {{$L^p$}-{$L^q$} estimate for wave equation with bounded time
              dependent coefficient},
   JOURNAL = {Hokkaido Math. J.},
  FJOURNAL = {Hokkaido Mathematical Journal},
    VOLUME = {34},
      YEAR = {2005},
    NUMBER = {3},
     PAGES = {541--586},
      ISSN = {0385-4035},
   MRCLASS = {35L05 (35B45 35L15)},
  MRNUMBER = {2186822},
MRREVIEWER = {Nasser-eddine\ Tatar},
       DOI = {10.14492/hokmj/1285766286},
       URL = {https://doi.org/10.14492/hokmj/1285766286},
}

@article {T07,
    AUTHOR = {Tarama, Shigeo},
     TITLE = {Energy estimate for wave equations with coefficients in some
              {B}esov type class},
   JOURNAL = {Electron. J. Differential Equations},
  FJOURNAL = {Electronic Journal of Differential Equations},
      YEAR = {2007},
     PAGES = {No. 85, 12},
      ISSN = {1072-6691},
   MRCLASS = {35L15 (35B45)},
  MRNUMBER = {2328686},
MRREVIEWER = {Silvano\ B.\ de Menezes},
}

@article {BF24,
    AUTHOR = {Banjai, Lehel and Ferrari, Matteo},
     TITLE = {Runge-{K}utta convolution quadrature based on {G}auss methods},
   JOURNAL = {Numer. Math.},
  FJOURNAL = {Numerische Mathematik},
    VOLUME = {156},
      YEAR = {2024},
    NUMBER = {5},
     PAGES = {1719--1750},
      ISSN = {0029-599X,0945-3245},
   MRCLASS = {65L06 (65M15 65R20)},
  MRNUMBER = {4811804},
MRREVIEWER = {Xiao\ Tang},
       DOI = {10.1007/s00211-024-01429-4},
       URL = {https://doi.org/10.1007/s00211-024-01429-4},
}

@article {BLM11,
    AUTHOR = {Banjai, Lehel and Lubich, Christian and Melenk, Jens Markus},
     TITLE = {Runge-Kutta convolution quadrature for operators arising in wave propagation},
   JOURNAL = {Numer. Math.},
  FJOURNAL = {Numerische Mathematik},
    VOLUME = {119},
      YEAR = {2011},
    NUMBER = {1},
     PAGES = {1--20},
      ISSN = {0029-599X,0945-3245},
   MRCLASS = {65D30 (65E05 65R10)},
  MRNUMBER = {2824853},
MRREVIEWER = {Michael\ J.\ Carley},
       DOI = {10.1007/s00211-011-0378-z},
       URL = {https://doi.org/10.1007/s00211-011-0378-z},
}

@inproceedings {GHL18,
    AUTHOR = {Gauckler, Ludwig and Hairer, Ernst and Lubich, Christian},
     TITLE = {Dynamics, numerical analysis, and some geometry},
 BOOKTITLE = {Proceedings of the {I}nternational {C}ongress of
              {M}athematicians---{R}io de {J}aneiro 2018. {V}ol. {I}.
              {P}lenary lectures},
     PAGES = {453--485},
 PUBLISHER = {World Sci. Publ., Hackensack, NJ},
      YEAR = {2018},
      ISBN = {978-981-3272-90-3; 978-981-3272-87-3},
   MRCLASS = {65P10 (37M15 81-04)},
  MRNUMBER = {3966736},
}

@article {G21,
    AUTHOR = {Garnier, Josselin},
     TITLE = {Wave propagation in periodic and random time-dependent media},
   JOURNAL = {Multiscale Model. Simul.},
  FJOURNAL = {Multiscale Modeling \& Simulation. A SIAM Interdisciplinary
              Journal},
    VOLUME = {19},
      YEAR = {2021},
    NUMBER = {3},
     PAGES = {1190--1211},
      ISSN = {1540-3459,1540-3467},
   MRCLASS = {78A48 (35Q61 35R60)},
  MRNUMBER = {4292306},
MRREVIEWER = {Zhiwei\ Fang},
       DOI = {10.1137/20M1377734},
       URL = {https://doi.org/10.1137/20M1377734},
}

@article {CCNM15,
    AUTHOR = {Chartier, Philippe and Crouseilles, Nicolas and Lemou,
              Mohammed and M\'ehats, Florian},
     TITLE = {Uniformly accurate numerical schemes for highly oscillatory
              {K}lein-{G}ordon and nonlinear {S}chr\"odinger equations},
   JOURNAL = {Numer. Math.},
  FJOURNAL = {Numerische Mathematik},
    VOLUME = {129},
      YEAR = {2015},
    NUMBER = {2},
     PAGES = {211--250},
      ISSN = {0029-599X,0945-3245},
   MRCLASS = {65M12 (35C20 35Q55 74Q10)},
  MRNUMBER = {3300419},
MRREVIEWER = {Istv\'an\ Farag\'o},
       DOI = {10.1007/s00211-014-0638-9},
       URL = {https://doi.org/10.1007/s00211-014-0638-9},
}

@article {CHL03,
    AUTHOR = {Cohen, David and Hairer, Ernst and Lubich, Christian},
     TITLE = {Modulated {F}ourier expansions of highly oscillatory
              differential equations},
   JOURNAL = {Found. Comput. Math.},
  FJOURNAL = {Foundations of Computational Mathematics. The Journal of the
              Society for the Foundations of Computational Mathematics},
    VOLUME = {3},
      YEAR = {2003},
    NUMBER = {4},
     PAGES = {327--345},
      ISSN = {1615-3375,1615-3383},
   MRCLASS = {34C15 (34E13 37J40 82C05)},
  MRNUMBER = {2009682},
MRREVIEWER = {Iulian\ Coroian},
       DOI = {10.1007/s10208-002-0062-x},
       URL = {https://doi.org/10.1007/s10208-002-0062-x},
}

@article {ACHR24,
    AUTHOR = {Ammari, Habib and Cao, Jinghao and Hiltunen, Erik Orvehed and
              Rueff, Liora},
     TITLE = {Scattering from time-modulated subwavelength resonators},
   JOURNAL = {Proc. A.},
  FJOURNAL = {Proceedings A},
    VOLUME = {480},
      YEAR = {2024},
    NUMBER = {2289},
     PAGES = {Paper No. 20240177, 22},
      ISSN = {1364-5021,1471-2946},
   MRCLASS = {35L51 (35P25 74J20)},
  MRNUMBER = {4755815},
}

@article{HR25,
  title={Energy Balance and Optical Theorem for Time-Modulated Subwavelength Resonator Arrays},
  author={Hiltunen, Erik Orvehed and Rueff, Liora},
  journal={arXiv preprint arXiv:2507.11201},
  year={2025}
}

@article{HD24,
  title = {Coupled harmonics due to time-modulated point scatterers},
  author = {Hiltunen, E. O. and Davies, B.},
  journal = {Phys. Rev. B},
  volume = {110},
  issue = {18},
  pages = {184102},
  numpages = {8},
  year = {2024},
  month = {Nov},
  publisher = {American Physical Society},
  doi = {10.1103/PhysRevB.110.184102},
  url = {https://link.aps.org/doi/10.1103/PhysRevB.110.184102}
}

@book {BS22,
    AUTHOR = {Banjai, Lehel and Sayas, Francisco-Javier},
     TITLE = {Integral equation methods for evolutionary {PDE}---a
              convolution quadrature approach},
    SERIES = {Springer Series in Computational Mathematics},
    VOLUME = {59},
 PUBLISHER = {Springer, Cham},
      YEAR = {[2022] \copyright 2022},
     PAGES = {xix+268},
      ISBN = {978-3-031-13219-3; 978-3-031-13220-9},
   MRCLASS = {65-01 (35L05 35L20 45A05 65D32 65M38 65N38 65R20)},
  MRNUMBER = {4807209},
       DOI = {10.1007/978-3-031-13220-9},
       URL = {https://doi.org/10.1007/978-3-031-13220-9},
}

@article {GHL16,
    AUTHOR = {Gauckler, Ludwig and Hairer, Ernst and Lubich, Christian},
     TITLE = {Long-term analysis of semilinear wave equations with slowly
              varying wave speed},
   JOURNAL = {Comm. Partial Differential Equations},
  FJOURNAL = {Communications in Partial Differential Equations},
    VOLUME = {41},
      YEAR = {2016},
    NUMBER = {12},
     PAGES = {1934--1959},
      ISSN = {0360-5302,1532-4133},
   MRCLASS = {35L71 (35A01 35B25 35C20 35L20)},
  MRNUMBER = {3572564},
MRREVIEWER = {Joseph\ L.\ Shomberg},
       DOI = {10.1080/03605302.2016.1235581},
       URL = {https://doi.org/10.1080/03605302.2016.1235581},
}

@article{NHA24,
  title={Wave scattering with time-periodic coefficients: Energy estimates and harmonic formulations},
  author={Nick, J{\"o}rg and Hiptmair, Ralf and Ammari, Habib},
  journal={arXiv preprint arXiv:2410.10297},
  year={2024}
}

@article{DV25,
  title={Local and nonlocal homogenization of wave propagation in time-varying media},
  author={D{\"o}ding, Christian and Verf{\"u}rth, Barbara},
  journal={arXiv preprint arXiv:2505.21390},
  year={2025}
}

@book {K93,
    AUTHOR = {Kuchment, Peter},
     TITLE = {Floquet theory for partial differential equations},
    SERIES = {Operator Theory: Advances and Applications},
    VOLUME = {60},
 PUBLISHER = {Birkh\"auser Verlag, Basel},
      YEAR = {1993},
     PAGES = {xiv+350},
      ISBN = {3-7643-2901-7},
   MRCLASS = {35-02 (35C15 35P10 47N20)},
  MRNUMBER = {1232660},
MRREVIEWER = {Yehuda\ Pinchover},
       DOI = {10.1007/978-3-0348-8573-7},
       URL = {https://doi.org/10.1007/978-3-0348-8573-7},
}

@article{Y17,
  title = {Frequency conversion induced by time-space modulated media},
  author = {Yi, Kaijun and Collet, Manuel and Karkar, Sami},
  journal = {Phys. Rev. B},
  volume = {96},
  issue = {10},
  pages = {104110},
  numpages = {8},
  year = {2017},
  month = {Sep},
  publisher = {American Physical Society},
  doi = {10.1103/PhysRevB.96.104110},
  url = {https://link.aps.org/doi/10.1103/PhysRevB.96.104110}
}

@article{S19,
  title = {Nonreciprocal acoustic transmission in space-time modulated coupled resonators},
  author = {Shen, Chen and Zhu, Xiaohui and Li, Junfei and Cummer, Steven A.},
  journal = {Phys. Rev. B},
  volume = {100},
  issue = {5},
  pages = {054302},
  numpages = {6},
  year = {2019},
  month = {Aug},
  publisher = {American Physical Society},
  doi = {10.1103/PhysRevB.100.054302},
  url = {https://link.aps.org/doi/10.1103/PhysRevB.100.054302}
}

@article{A22,
  title = {Parametric Mie Resonances and Directional Amplification in Time-Modulated Scatterers},
  author = {Asadchy, V. and Lamprianidis, A.G. and Ptitcyn, G. and Albooyeh, M. and Rituraj and Karamanos, T. and Alaee, R. and Tretyakov, S.A. and Rockstuhl, C. and Fan, S.},
  journal = {Phys. Rev. Appl.},
  volume = {18},
  issue = {5},
  pages = {054065},
  numpages = {8},
  year = {2022},
  month = {Nov},
  publisher = {American Physical Society},
  doi = {10.1103/PhysRevApplied.18.054065},
  url = {https://link.aps.org/doi/10.1103/PhysRevApplied.18.054065}
}

@book {WYW13,
    AUTHOR = {Wu, Xinyuan and You, Xiong and Wang, Bin},
     TITLE = {Structure-preserving algorithms for oscillatory differential
              equations},
 PUBLISHER = {Springer, Heidelberg; Science Press Beijing, Beijing},
      YEAR = {2013},
     PAGES = {xii+236},
      ISBN = {978-3-642-35337-6; 978-3-642-35338-3; 978-7-03-035520-1},
   MRCLASS = {65-02 (65Lxx)},
  MRNUMBER = {3026657},
MRREVIEWER = {Martin\ Hermann},
       DOI = {10.1007/978-3-642-35338-3},
       URL = {https://doi.org/10.1007/978-3-642-35338-3},
}

@article {FA24,
    AUTHOR = {Feppon, F. and Ammari, H.},
     TITLE = {Subwavelength resonant acoustic scattering in fast
              time-modulated media},
   JOURNAL = {J. Math. Pures Appl. (9)},
  FJOURNAL = {Journal de Math\'ematiques Pures et Appliqu\'ees. Neuvi\`eme
              S\'erie},
    VOLUME = {187},
      YEAR = {2024},
     PAGES = {233--293},
      ISSN = {0021-7824,1776-3371},
   MRCLASS = {35B34 (35B10 35B40 35L05 45M05)},
  MRNUMBER = {4756020},
       DOI = {10.1016/j.matpur.2024.05.012},
       URL = {https://doi.org/10.1016/j.matpur.2024.05.012},
}
\end{document}